\documentclass[11pt]{amsart}
\usepackage{amssymb,amscd,amsthm,verbatim}
\usepackage{amsbsy}
\usepackage{latexsym}
\usepackage[all,cmtip]{xy}
\usepackage[mathscr]{euscript}

\usepackage{etex}
\usepackage{graphicx}
\usepackage{pictex}
\usepackage{epstopdf}
\usepackage{color}

\usepackage{bbm}

\date{\today}

\newcommand{\Z}{{\mathbb Z}}

\newcommand{\D}{{\mathbb D}}

\newcommand{\N}{{\mathbb N}}

\renewcommand{\Re}{{\mathrm{Re} \,}}

\newtheorem{theorem}{Theorem}[section]
\newtheorem{lemma}[theorem]{Lemma}
\newtheorem{prop}[theorem]{Proposition}

\theoremstyle{definition}

\newtheorem{remark}[theorem]{Remark}

\theoremstyle{plain}

\allowdisplaybreaks \numberwithin{equation}{section}

\DeclareMathOperator{\supp}{supp}

\begin{document}

\title{Simon's OPUC Hausdorff Dimension Conjecture}

\author[D.\ Damanik]{David Damanik}
\address{Department of Mathematics, Rice University, Houston, TX~77005, USA}
\email{damanik@rice.edu}
\thanks{D.D.\ was supported in part by NSF grant DMS--1700131 and by an Alexander von Humboldt Foundation research award}

\author[S.\ Guo]{Shuzheng Guo}
\address{Ocean University of China, Qingdao 266100, Shandong, China and Rice University, Houston, TX~77005, USA}
\email{gszouc@gmail.com}
\thanks{S.G.\ was supported by CSC (No. 201906330008) and NSFC (No. 11571327)}

\author[D.\ Ong]{Darren C.\ Ong}
\address{Department of Mathematics, Xiamen University Malaysia, 43900 Sepang, Selangor Darul Ehsan, Malaysia}
\email{darrenong@xmu.edu.my}
\thanks{D.O.\ was supported in part by a grant from the Fundamental Research Grant Scheme from the Malaysian Ministry of Education (Grant No: FRGS/1/2018/STG06/XMU/02/1) and a Xiamen University Malaysia Research Fund (Grant Number: XMUMRF/2020-C5/IMAT/0011)}

\begin{abstract}
We show that the Szeg\H{o} matrices, associated with Verblunsky coefficients $\{\alpha_n\}_{n\in\Z_+}$ obeying $\sum_{n = 0}^\infty n^\gamma |\alpha_n|^2 < \infty$ for some $\gamma \in (0,1)$, are bounded for values $z \in \partial \D$ outside a set of Hausdorff dimension no more than $1 - \gamma$. In particular, the singular part of the associated probability measure on the unit circle is supported by a set of Hausdorff dimension no more than $1-\gamma$. This proves the OPUC Hausdorff dimension conjecture of Barry Simon from 2005.
\end{abstract}

\maketitle

\section{Introduction}

This paper is concerned with the correspondence between non-trivial probability measures on the unit circle and their associated Verblunsky coefficients. This area has seen a large amount of activity in the past two decades, primarily due to the monographs \cite{S04, S05} by Barry Simon to which we refer the reader for general background material. In these monographs, Simon makes a number of conjectures, which are listed for the reader's convenience in \cite[Appendix~D]{S05}. Our main purpose here is to prove one of these conjectures, namely the so-called \emph{Hausdorff dimension conjecture}, which is stated in \cite[Conjecture~10.12.10]{S05}.

Before stating our result, let us describe the setting. Suppose $\mu$ is a non-trivial (i.e., not finitely supported) probability measure on the unit circle $\partial \mathbb{D} = \{ z \in \mathbb{C} : |z| = 1 \}$. By the non-triviality assumption, the functions $1, z, z^2, \cdots$ are linearly independent in the Hilbert space $\mathcal{H} = L^2(\partial\mathbb{D}, d\mu)$, and hence one can form, by the Gram-Schmidt procedure, the \emph{monic orthogonal polynomials} $\Phi_n(z)$, whose \emph{Szeg\H{o} dual} is defined by $\Phi_n^{*} = z^n\overline{\Phi_n({1}/{\overline{z}})}$. There are constants $\{\alpha_n\}_{n \in \Z_+}$ in $\mathbb{D}=\{z\in\mathbb{C}:|z|<1\}$, called the \emph{Verblunsky coefficients}, so that
\begin{equation}\label{eq01}
\Phi_{n+1}(z) = z \Phi_n(z) - \overline{\alpha}_n \Phi_n^*(z), \qquad \textrm{ for } n \in \Z_+,
\end{equation}
which is the so-called \emph{Szeg\H{o} recurrence}. Conversely, every sequence $\{\alpha_n\}_{n \in \Z_+}$ in $\mathbb{D}$ arises as the sequence of Verblunsky coefficients for a suitable nontrivial probability measure on $\partial\mathbb D$.

If we consider instead the \emph{orthonormal polynomials}
$$
\varphi(z, n)=\frac{\Phi_n(z)}{\|\Phi_n(z)\|_{\mu}},
$$
where $\|\cdot\|_{\mu}$ is the norm of $\mathcal{H}$, it is easy to see that \eqref{eq01} becomes
\begin{equation}\label{eq01b}
\rho_n \varphi(z, n+1 ) = z \varphi (z , n) - \overline{\alpha}_n \varphi^*(z , n), \textrm{ for } n\in \N,
\end{equation}
where $\rho_n = (1 - |\alpha_n|^2)^{1/2}$.

The Szeg\H{o} recurrence can be written in a matrix form as follows:
\begin{equation*}
\left(
\begin{matrix}
\varphi(z , n+1)\\
\varphi^{*}(z , n+1)
\end{matrix}
\right)
=
\frac{1}{\rho_n}
\left(
\begin{matrix}
z & -\overline{\alpha}_n\\
-\alpha_n z & 1
\end{matrix}
\right)
\left(
\begin{matrix}
\varphi(z , n)\\
\varphi^{*}(z , n)
\end{matrix}
\right), \textrm{ for } n\in\N.
\end{equation*}
Alternatively, one can consider a different initial condition and derive the \emph{orthogonal polynomials of the second kind}, by setting $\psi(z , 0) = 1$ and then
\begin{equation*}
\left(
\begin{matrix}
\psi(z , n+1)\\
- \psi^{*}(z , n+1)
\end{matrix}
\right)
=
\frac{1}{\rho_n}
\left(
\begin{matrix}
z & -\overline{\alpha}_n\\
-\alpha_n z & 1
\end{matrix}
\right)
\left(
\begin{matrix}
\psi(z , n)\\
- \psi^{*}(z , n)
\end{matrix}
\right), \textrm{ for } n\in\N.
\end{equation*}
Define the associated \emph{Szeg\H{o} matrices}
$$
T_n(z)=\frac{1}{2}
\left(
\begin{matrix}
\varphi_n(z) + \psi_n(z) & \varphi_n(z) - \psi_n(z)\\
\varphi^*_n(z) - \psi^*_n(z) & \varphi^*_n(z) + \psi^*_n(z)
\end{matrix}
\right), \textrm{ for } n\in\N.
$$

We will assume that the Verblunsky coefficients obey a decay estimate of the form
\begin{equation}\label{e.vcass}
\sum_{n = 0}^\infty n^\gamma |\alpha_n|^2 < \infty
\end{equation}
for some $\gamma \in (0,1)$. Simon conjectured in \cite{S05} that under this decay assumption, the Szeg\H{o} matrices are bounded for values $z \in \partial \D$ outside a set of Hausdorff dimension no more than $1 - \gamma$; compare \cite[Conjecture~10.12.10]{S05}. The primary interest in such a statement comes from the general principle which says that the singular part of $\mu$ is supported by the set of those $z \in \partial \D$ for which the associated Szeg\H{o} matrices are unbounded, that is, such a result will imply that the singular part of the measure, $\mu_\mathrm{sing}$, has a support of Hausdorff dimension at most $1 - \gamma$.

The goal of this paper is to prove this conjecture:

\begin{theorem}\label{t.main}
Suppose that $\mu$ is such that the associated Verblunsky coefficients satisfy \eqref{e.vcass} for $\gamma \in (0,1)$. Then there is a set $S \subset \partial \D$ of Hausdorff dimension at most $1 - \gamma$ so that for $z \in \partial \D \setminus S$,
$$
\sup_{n \ge 0} \|T_n(z)\| < \infty.
$$
In particular, $\mu_\mathrm{sing}$ is supported by a set of dimension at most $1 - \gamma$.
\end{theorem}

\begin{remark}
(a) A useful perspective on this result is that it establishes a natural interpolation between the known endpoint results. The endpoint $0$ corresponds to the famous Szeg\H{o} theorem: formulated as \cite[Theorem~1.8.6]{S11}, Verblunsky's form of Szeg\H{o}'s theorem reads
\begin{equation}\label{e.szthm1}
\prod_{n = 0}^\infty (1 - |\alpha_n|^2) = \exp \left( \int \log (w(\theta)) \, \frac{d\theta}{2\pi} \right),
\end{equation}
where $w$ denotes the Radon-Nikodym derivative of the absolutely continuous part of $\mu$ with respect to the normalized Lebesgue measure on the unit circle. The identity \eqref{e.szthm1} implies in particular that
\begin{equation}\label{e.szthm2}
\sum_{n = 0}^\infty |\alpha_n|^2 < \infty \; \Leftrightarrow \; \int \log (w(\theta)) \, \frac{d\theta}{2\pi} > - \infty.
\end{equation}
Observe that the singular part of $\mu$ is entirely unrestricted by \eqref{e.szthm2} (other than by having weight less than $1$), and hence one can have measures $\mu$ whose Verblunsky coefficients satisfy \eqref{e.vcass} with $\gamma = 0$ and singular part of Hausdorff dimension $1 = 1 - \gamma$.

On the other hand, if $\mu$ is such that its Verblunsky coefficients satisfy \eqref{e.vcass} with $\gamma = 1$, then by Golinskii-Ibragimov \cite{GI71} (cf.~\cite[Theorem~6.1.2]{S04}; see also \cite{D06, DK04, S03} for extensions), $\mu_\mathrm{sing}$ vanishes, and in particular it is supported by a set of dimension $0 = 1 - \gamma$.

(b) The upper bound $1 - \gamma$ for the Hausdorff dimension of a support of $\mu_\mathrm{sing}$ cannot be improved (as proved by Denisov and Kupin in \cite{DK05} and presented in detail for the OPUC case at hand in \cite[Section~2.12]{S04}), and hence the result is optimal.

(c) Simon's conjecture \cite[Conjecture~10.12.10]{S05} is motivated by existing results in the Schr\"odinger operator literature and the well-known analogy between the OPUC theory and the theory of discrete one-dimensional Schr\"odinger operators (and more generally Jacobi matrices, which are standard OPRL models). Here, OPUC is short for \emph{orthogonal polynomials on the unit circle} and OPRL is short for \emph{orthogonal polynomials on the real line}. Specifically, there are results due to Remling \cite{r, r2} and Christ-Kiselev \cite{ck} in the Schr\"odinger operator literature that establish analogous dimension bounds for the singular part of a spectral measure under suitable assumptions on the potential of the Schr\"odinger operator in question.

(d) While the analogy between OPUC and OPRL is well documented, there is no canonical way of transforming a result and proof on one side into a result and proof on the other side. However, many of the tools used in a proof usually have analogs on the other side and one can attempt to carry out the same proof strategy. This works sometimes easily, sometimes with major effort, and sometimes not at all. Indeed, the second part of Simon's monograph, \cite{S05}, is written in this spirit: it works out a plethora of OPUC analogs of existing OPRL results. The conjecture at hand is a case where it does not work easily, and the conjecture is really asking whether it can be made to work with the appropriate amount of effort.

(d) The issue with the existing Schr\"odinger/OPRL results \cite{ck, r, r2} is that Remling works under a stronger assumption and that Christ-Kiselev use methods for which it is less clear what the appropriate OPUC analog might be. Remling assumes a pointwise estimate -- the suitable analog replacing \eqref{e.vcass} would be
\begin{equation}\label{e.vcass.pw}
|\alpha_n| \le \frac{C}{(1 + n)^\delta}.
\end{equation}
Here the known endpoints are $\delta = \frac12$ and $\delta = 1$, and the optimal dimension estimate reads $\mathrm{dim}_\mathrm{H} (S) \le 2 (1-\delta)$. However, Remling's proof uses the stronger assumption \eqref{e.vcass.pw} in a crucial way and it is not clear if his approach can be carried out under the weaker assumption \eqref{e.vcass}. On the other hand, Remling's approach employs Pr\"ufer variables, which is a concept that exists in both the OPRL setting and the OPUC setting, and Simon explicitly writes in \cite{S05} that he expects that ``using Pr\"ufer variables, one should be able to prove'' \cite[Conjecture~10.12.10]{S05}. Consequently, we have chosen to use Pr\"ufer variables in our attempt to prove \cite[Conjecture~10.12.10]{S05} and the general structure of our proof of Theorem~\ref{t.main} is inspired by Remling's work \cite{r, r2}. Of course we need to address the difficulties arising from not having a pointwise estimate of the form \eqref{e.vcass.pw}.

(e) What comes to our rescue in the OPUC case is the fact that a certain non-oscillatory behavior in the Pr\"ufer equations for the Schr\"odinger operator case is absent from the OPUC Pr\"ufer equations; compare Remark~\ref{r.opucsodiff} below. In particular, it remains an interesting problem to determine whether the Christ-Kiselev result can be obtained via Remling's method.
\end{remark}

This paper is organized as follows. We first recall the concept of Pr\"{u}fer variables in Section~\ref{sec.2}. Section~\ref{sec.3} establishes norm estimates for what may be called the WKB transform: a generalization of the Fourier transform involving Pr\"ufer variables. These estimates are then used in Section~\ref{sec.4} to carry out an iteration procedure that reduces the desired finiteness statement to one on a sufficiently refined partition of $\Z_+$, which can then be verified. Given this work, we are then able to quickly deduce Theorem~\ref{t.main} in Section~\ref{sec.5}.

\section{Pr\"ufer Variables}\label{sec.2}

In this section we recall the concept of Pr\"ufer variables. We refer the reader to \cite[Section 10.12]{S05} for the statements given below (see especially \cite[Theorems~10.12.1 and 10.12.3]{S05}) and for more information.

Let $\{\alpha_n\}_{n \in \Z_+}$ be the Verblunsky coefficients of a nontrivial probability measure $d\mu$ on $\partial \D$. As mentioned above, the $\alpha_n$'s give rise to a sequence $\{\Phi_n(z)\}_{n \in \Z_+}$ of monic polynomials (via the Szeg\H{o} recurrence \eqref{eq01}) that are orthogonal with respect to $d\mu$. For $\beta \in [0,2\pi)$, we also consider the monic polynomials $\{\Phi_n(z,\beta)\}_{n\in\Z_+}$ that are associated in the same way with the Verblunsky
coefficients $\{ e^{i\beta} \alpha_n \}_{n\in\Z_+}$. The parameter $\beta$ corresponds to a variation of the initial condition for the Szeg\H{o} recursion. In particular, the orthogonal polynomials of both the first and second kind arise for suitable choices of $\beta$, and hence we can bound the Szeg\H{o} matrices once we have bounds for these two relevant values of $\beta$.

Let $\eta \in [0,2\pi)$. Define the Pr\"ufer variables $R_n, \theta_n$ by
\begin{equation}\label{e.pruferdef}
\Phi_n(e^{i\eta},\beta) = R_n(\eta,\beta) \exp \left[ i(n \eta + \theta_n(\eta,\beta))
\right],
\end{equation}
where $R_n > 0$, $\theta_n \in [0,2\pi)$, and $|\theta_{n+1} - \theta_n| < \pi$. Thus, our goal of bounding the $\beta$-dependent orthogonal polynomials can be accomplished by bounding the $\beta$-dependent Pr\"ufer radius $R_n(\eta,\beta)$ as $n \to \infty$ outside a set of $\eta$'s that has sufficiently small Hausdorff dimension.

The Pr\"ufer variables obey the following pair of equations:
\begin{align}
\label{eq:RPrufer} \frac{R_{n+1}^2(\eta,\beta)}{R_n^2(\eta,\beta)} & = 1 + | \alpha_n |^2 - 2 \Re \left(\alpha_n e^{i[(n+1)\eta + \beta + 2 \theta_n(\eta,\beta)]} \right), \\
\label{eq:ThetaPrufer} e^{-i(\theta_{n+1}(\eta,\beta) - \theta_n(\eta,\beta))} & = \frac{1 - \alpha_n e^{i[(n+1)\eta + \beta + 2 \theta_n(\eta,\beta)]}}{\left[1 + | \alpha_n |^2 - 2 \Re \left(\alpha_n e^{i[(n+1)\eta + \beta + 2 \theta_n(\eta,\beta)]} \right)\right]^{1/2}}.
\end{align}
We also define $r_n(\eta,\beta) = |\varphi_n(\eta,\beta)|$.

\begin{remark}\label{r.opucsodiff}
For the Schr\"odinger differential equation
$$
-y''(x) + V(x) y(x) = E y(x), \quad x > 0
$$
with $E = k^2 > 0$, the Pr\"ufer variables defined by
$$
\begin{pmatrix} \frac{1}{k} y'(x) \\ y(x) \end{pmatrix} = R(x) \begin{pmatrix} \cos \frac{\psi(x)}{2} \\ \sin \frac{\psi(x)}{2} \end{pmatrix}
$$
obey the pair of equations
\begin{align*}
\frac{d}{dx} \ln R (x) & = \frac{V(x)}{2k} \sin \psi(x) , \\
\frac{d}{dx} \psi (x) & = 2 k - \frac{V(x)}{k} + \frac{V(x)}{k} \cos \psi (x).
\end{align*}
Note in particular that in the Schr\"odinger case, the derivative of the Pr\"ufer angle contains a non-oscillating term that is linear in the potential $V$, whereas \eqref{eq:ThetaPrufer} contains no non-oscillating term that is linear in $\alpha$.
\end{remark}

When $\{ \alpha_n \} \in \ell^2$, we have
\begin{equation}\label{fs}
r_n(\eta,\beta) \sim R_n(\eta,\beta) \sim \exp \left( - \sum_{j=0}^{n-1} \Re (\alpha_j e^{i[(j+1)\eta + \beta + 2 \theta_j(\eta,\beta)]}) \right),
\end{equation}
where here and in the remainder of this paper we use the following notation: we write $f \lesssim g$ (resp., $f \gtrsim g$) for non-negative quantities $f,g$ if there is a constant $C > 0$ that is uniform in all parameters that are variable in the situation at hand such that $f \le C g$ (resp., $f \ge C g$). If we have both $f \lesssim g$ and $f \gtrsim g$, we write $f \sim g$.


\section{Norm Estimates for the WKB Transform}\label{sec.3}

Let $\nu$ denote a finite Borel measure whose support is contained in some compact subset of $(0,2\pi]$. Such a $\nu$ is said to be (uniformly) $D$-dimensional, with $D\in [0,1]$, if $\nu(I)\le C\left|I\right|^D$ for all intervals $I\subset (0,2\pi]$. Let us write
\begin{align}
\psi(k,\eta,\beta) & = (k+1) \eta + \beta + 2 \theta_n (\eta,\beta), \label{e.psidef} \\
\omega(s,\eta,\beta) & = (s+1)\eta +\beta +\frac{1}{\eta}\sum_{k=0}^{s-1} |\alpha_k|^2. \label{e.omegadef}
\end{align}

\begin{theorem}{\label{t2.1}}
Assume that \eqref{e.vcass} holds. Suppose that $\nu$ is a finite Borel measure on $(0,2\pi)$ with the following two properties:
\begin{itemize}

\item[{\rm (i)}] There is a $\delta > 0$ such that $\nu$ is supported by $(\delta,2\pi-\delta)$.

\item[{\rm (ii)}] There is a $D \in (1-\gamma,1)$ such that $\nu$ is uniformly $D$-H\"older continuous, that is, $\nu(I) \lesssim |I|^D$ for every interval $I \subseteq (0,2\pi)$.

\end{itemize}
Then
\begin{equation}\label{e.WKBTestimate}
\sup_\beta \sum_{s=0}^{L} \,  \left| \int f(\eta) e^{i\omega(s,\eta,\beta)} \, d\nu(\eta) \right|^2 \lesssim (L+1)^{1-D} \int  \left| f(\eta) \right|^2 \, d\nu(\eta),
\end{equation}
for all $f\in L^2((0,2\pi),d\nu)$ and $L \in \Z_+$.
\end{theorem}

\begin{remark}
(a) In analogy with \cite{r, r2} and many earlier works, one may call the quantity
$$
\int f(\eta) e^{i\omega(s,\eta,\beta)} \, d\nu(\eta)
$$
on the left-hand side of \eqref{e.WKBTestimate} the \emph{WKB transform} of $f$. Referring to the definition of $\omega(s,\eta,\beta)$ in \eqref{e.omegadef}, we note that it is a generalized version of the Fourier transform of $f$.
\\[2mm]
(b) In the proof of Theorem~\ref{t2.1}, we will use the summation by parts formulae,
\begin{equation}\label{e.sumbyparts}
\sum_{j = J_\ell}^{J_r} a_j (b_{j+1} - b_j) = (a_{J_r} b_{J_r + 1} - a_{J_\ell} b_{J_\ell}) - \sum_{j = J_\ell + 1}^{J_r} b_j (a_{j} - a_{j - 1}),
\end{equation}
and
\begin{equation}\label{e.sumbyparts2}
\sum_{j = J_\ell}^{J_r} a_j b_j = a_{J_{\ell}} \sum_{j = J_\ell}^{J_r} b_j + \sum_{j = J_\ell}^{J_r-1} (a_{j+1} - a_j) \sum_{k=j+1}^{J_r} b_k.
\end{equation}
If all quantities have a limit as $J_r \to \infty$ (resp., $J_\ell \to -\infty$), the formulae extend to the case $J_r = \infty$ (resp., $J_\ell = -\infty$).
\end{remark}

\begin{proof}[Proof of Theorem~\ref{t2.1}]
The case $L=0$ follows immediately from the Cauchy-Schwarz inequality and finiteness of $\nu$. We can therefore restrict our attention to the case
\begin{equation}{\label{e.remainingcase}}
L\ge 1.
\end{equation}
We have
\begin{align}
\sum_{s=0}^L & \left|\int f(\eta) e^{i\omega(s,\eta,\beta)} \, d\nu(\eta) \right|^2 \notag\\
& \le \sum_{s=0}^{L} e^{\frac{L^2 - s^2}{L^2}} \left|\int f(\eta) e^{i\omega(s,\eta,\beta)} \, d\nu(\eta) \right|^2 \text{ since } s\le L \notag \\
& \le \sum_{s=-\infty}^{\infty} e^{\frac{L^2 - s^2}{L^2}} \left|\int f(\eta) e^{i\omega(s,\eta,\beta)} \, d\nu(\eta) \right|^2 \text{ since } e^{-x}>0 \notag \\
& = e \sum_{s=-\infty}^{\infty} e^{-\frac{s^2}{L^2}} \left|\int f(\eta) e^{i\omega(s,\eta,\beta)} \, d\nu(\eta) \right|^2 \notag \\
& = e \iint f(\eta) \overline{f(\eta')} \sum_{s=-\infty}^{\infty} e^{-\frac{s^2}{L^2}} e^{i(\omega(s,\eta,\beta)-\omega(s,\eta',\beta))} \, d\nu(\eta) \, d\nu(\eta') \notag \\
& = e \iint \, d\nu(\eta) \, d\nu(\eta') f(\eta) \overline{f(\eta')} \times \notag \\
& \qquad \qquad \times \sum_{s=-\infty}^{\infty} e^{-\frac{s^2}{L^2}} e^{i((s+1)(\eta-\eta') + (\frac{1}{\eta} - \frac{1}{\eta'}) \sum_{k=0}^{s-1} |\alpha_k|^2)}   \notag \\
& = e \iint \, d\nu(\eta) \, d\nu(\eta')f(\eta) \overline{f(\eta')} e^{i(\eta-\eta')} \times \notag \\
& \qquad \qquad \times \sum_{s=-\infty}^{\infty} e^{-\frac{s^2}{L^2}} e^{i(s(\eta-\eta') + (\frac{1}{\eta} - \frac{1}{\eta'}) \sum_{k=0}^{s-1} |\alpha_k|^2) } \label{estimate5}
\end{align}
where we set
\begin{equation}\label{e.negativealphas}
\alpha_n = 0 \text{ for } n \in \Z_-.
\end{equation}

In the $s$-summation, we apply \eqref{e.sumbyparts2} with
$$
a_j = e^{i(\frac{1}{\eta}-\frac{1}{\eta'}) \sum_{k=0}^{j-1} |\alpha_k|^2}, \quad b_j = e^{-\frac{j^2}{L^2}} e^{ij(\eta-\eta')}, \quad J_\ell = -\infty, \quad J_r = \infty.
$$
and obtain
\begin{align}
\sum_{s=-\infty}^{\infty} & e^{-\frac{s^2}{L^2}} e^{i(s(\eta-\eta') + (\frac{1}{\eta}- \frac{1}{\eta'} ) \sum_{k=0}^{s-1} |\alpha_k|^2) } \notag \\
& = \sum_{s=-\infty}^{\infty} e^{-\frac{s^2}{L^2}} e^{is(\eta-\eta')} \label{estimate6a}\\
& \quad + \sum_{s=-\infty}^{\infty} \big(e^{i(\frac{1}{\eta}  - \frac{1}{\eta'}) \sum_{k=0}^{s} | \alpha_k|^2 }  - e^{i(\frac{1}{\eta} - \frac{1}{\eta'}) \sum_{k=0}^{s-1} |\alpha_k|^2 } \big) \sum_{j=s+1}^{\infty} e^{-\frac{j^2}{L^2}} e^{ij(\eta-\eta')}.  \label{estimate6b}
\end{align}
The first term, \eqref{estimate6a}, can be evaluated as follows. By the Poisson summation formula, we have
\[
\sum_{s=-\infty}^{\infty} e^{-\frac{s^2}{L^2}} e^{is(\eta-\eta')} = \sum_{k=-\infty}^{\infty} c_k,
\]
where
\begin{align*}
c_k & = \int_{-\infty}^{\infty} e^{-\frac{x^2}{L^2}} e^{ix(\eta-\eta')-2\pi i kx} \, dx \\
    & = \int_{-\infty}^{\infty} e^{-\frac{x^2}{L^2}+ix((\eta-\eta')-2\pi k)} \, dx\\
    & = \int_{-\infty}^{\infty} e^{-\frac{1}{L^2}(x^2 + ixL^2((\eta'-\eta)+2\pi k))} \, dx\\
    & = e^{-\frac{L^2((\eta'-\eta)+2\pi k)^2}{4}} \int_{-\infty}^{\infty} e^{-\frac{1}{L^2}(x + i\frac{L^2((\eta'-\eta)+2\pi k)}{2})^2} \, dx  \\
    & =  e^{-\frac{L^2((\eta'-\eta)+2\pi k)^2}{4}} \int_{-\infty}^{\infty} e^{-\frac{u^2}{L^2}} du\\
    & = C L e^{-\frac{L^2((\eta'-\eta)+2\pi k)^2}{4}}.
\end{align*}
Thus, we have
\[
\eqref{estimate6a} = \sum_{s=-\infty}^{\infty} e^{-\frac{s^2}{L^2}} e^{is(\eta-\eta')} = C L \sum_{k=-\infty}^{\infty} e^{-\frac{L^2 ((\eta'-\eta) + 2 \pi k)^2}{4}}.
\]
After inserting this into \eqref{estimate5}, we get
\begin{align*}
C L \iint & | f(\eta)\overline{f(\eta')}|  \sum_{k=-\infty}^{\infty} e^{-\frac{L^2((\eta' - \eta) + 2 \pi k)^2}{4}} \,  d\nu(\eta) \, d\nu(\eta')  \\
& \le C L \int |f(\eta)|^2 \int \sum_{k=-\infty}^{\infty} e^{-\frac{L^2 (\eta' - \eta + 2 \pi k)^2}{4}} \, d\nu(\eta') \, d\nu(\eta)
\end{align*}
Now,
\begin{align*}
\int \sum_{k=-\infty}^{\infty} & e^{-\frac{L^2(\eta'-\eta+2\pi k)^2}{4}} \, d\nu(\eta') \\
& \lesssim \sum_{k=0}^{\infty} \sum_{n=0}^{\lfloor 2 \pi L \rfloor - 1} \int_{\frac{n}{L} \le |\eta' - \eta| \le \frac{n+1}{L}} e^{-\frac{L^2 (\eta' - \eta + 2 k \pi)^2}{4}} \, d\nu(\eta') \\
& \qquad + \sum_{k=0}^{\infty} \int_{\frac{\lfloor 2\pi L\rfloor}{L} \le |\eta' - \eta| < 2\pi} e^{-\frac{L^2 (\eta' - \eta + 2 k \pi)^2}{4}} \, d\nu(\eta')\\
& \le L^{-D} \left( \sum_{k=0}^{\infty} \sum_{n=0}^{\lfloor 2 \pi L \rfloor - 1} e^{\frac{-(n + 2 k L \pi)^2}{4}} + \sum_{k=0}^{\infty} \sum_{n=0}^{\lfloor 2 \pi L \rfloor - 1} e^{\frac{-(-n - 1 + 2 k L \pi )^2}{4}}\right) \\
& \qquad + (2\pi- \frac{\lfloor 2\pi L \rfloor}{L})^{D}\left(\sum_{k=0}^{\infty}e^{-\frac{(\lfloor 2\pi L \rfloor + 2k \pi L)^2}{4}} + \sum_{k=0}^{\infty}e^{-\frac{(-2 \pi + 2k \pi L)^2}{4}} \right)\\
& \lesssim L^{-D} \left( \sum_{k=0}^{\infty} \sum_{n=0}^{\lfloor 2 \pi L \rfloor - 1} e^{-\frac{n^2}{4} - k^2 L^2 \pi^2} + C \sum_{m=0}^{\infty} e^{-\frac{m^2}{4}} \right.\\
& \qquad \left. + \sum_{k=0}^{\infty}e^{-\frac{(\lfloor 2\pi L \rfloor + 2k \pi L)^2}{4}} + \sum_{k=0}^{\infty}e^{-\frac{(-2 \pi + 2k \pi L)^2}{4}}\right)\\
& \lesssim  L^{-D},
\end{align*}
so, combining the three steps, we get the desired bound on the first ``half'' of \eqref{estimate5}.

It remains to study the second ``half,'' that is, the contribution coming from \eqref{estimate6b}. The basic strategy is as follows. We will break up the double integral over $\eta,\eta'$ into two parts, according to the size of $\left|\eta-\eta'\right|$. Since the $j$-summation from \eqref{estimate6b} is oscillatory, we expect that it decays as $\left|\eta-\eta'\right|$ grows. 
For small $\left|\eta-\eta'\right|$, there are no such cancelations, but then the $\eta,\eta'$ integration is only over a small region and we can use our assumptions on $\nu$.

Note that \eqref{estimate6b} can be written as
\begin{align}
\sum_{s=-\infty}^{\infty} &\big(e^{i(\frac{1}{\eta} - \frac{1}{\eta'}) \sum_{k=0}^{s} |\alpha_k|^2 } - e^{i(\frac{1}{\eta} - \frac{1}{\eta'}) \sum_{k=0}^{s-1} |\alpha_k|^2} \big) \sum_{j=s+1}^{\infty} e^{-\frac{j^2}{L^2}} e^{ij(\eta-\eta')}  \notag \\
& =  \sum_{s=0}^{\infty} \big(e^{i(\frac{1}{\eta} - \frac{1}{\eta'}) \sum_{k=0}^{s} |\alpha_k|^2 } - e^{i(\frac{1}{\eta}- \frac{1}{\eta'}) \sum_{k=0}^{s-1} |\alpha_k|^2}\big) \sum_{j=s+1}^{\infty} e^{-\frac{j^2}{L^2}} e^{ij(\eta-\eta')} \notag \\
& = \sum_{s=0}^{\infty} e^{i(\frac{1}{\eta}- \frac{1}{\eta'}) \sum_{k=0}^{s-1} |\alpha_k|^2 }\left(e^{i(\frac{1}{\eta}- \frac{1}{\eta'}) |\alpha_s|^2}-1\right)\sum_{j=s+1}^{\infty} e^{-\frac{j^2}{L^2}} e^{ij(\eta-\eta')} \notag \\
& = \sum_{s=0}^{\infty} ie^{i(\frac{1}{\eta}- \frac{1}{\eta'}) \sum_{k=0}^{s-1} |\alpha_k|^2 }\Big(\frac{1}{\eta}-\frac{1}{\eta'}\Big) |\alpha_s|^2 \sum_{j=s+1}^{\infty} e^{-\frac{j^2}{L^2}} e^{ij(\eta-\eta')} \label{estimate6b2} \\
\qquad & - \frac{1}{2!} \sum_{s=0}^{\infty} e^{i(\frac{1}{\eta}- \frac{1}{\eta'}) \sum_{k=0}^{s-1} |\alpha_k|^2 }\Big(\frac{1}{\eta}-\frac{1}{\eta'}\Big)^2 |\alpha_s|^4 \sum_{j=s+1}^{\infty} e^{-\frac{j^2}{L^2}} e^{ij(\eta-\eta')} \label{estimate6b3}\\
\qquad & - \frac{i}{3!} \sum_{s=0}^{\infty} e^{i(\frac{1}{\eta}- \frac{1}{\eta'}) \sum_{k=0}^{s-1} |\alpha_k|^2 }\Big(\frac{1}{\eta}-\frac{1}{\eta'}\Big)^3 |\alpha_s|^6 \sum_{j=s+1}^{\infty} e^{-\frac{j^2}{L^2}} e^{ij(\eta-\eta')} \label{estimate6b4}\\
\qquad & + \cdots \notag
\end{align}
where we used \eqref{e.negativealphas} in the first step and Taylor expansion in the last step. If one can get the desired bound after inserting \eqref{estimate6b2} into \eqref{estimate5}, one can also obtain $\frac{1}{2!}$ times the desired bound after inserting \eqref{estimate6b3} into \eqref{estimate5}. Similarly, one can also obtain $\frac{1}{3!}$ times the desired bound after inserting \eqref{estimate6b4} into \eqref{estimate5}, and so on. Then summing all these terms, and using that $\sum_{k=1}^{\infty} \frac{1}{k!} = e-1$, we get a constant times the desired bound. Thus, it suffices to consider the contribution coming from \eqref{estimate6b2}.

\medskip

We start by summing by parts in the $j$-summation:
\begin{align}
\sum_{j=s+1}^{\infty} & e^{-\frac{j^2}{L^2}} e^{ij(\eta-\eta')} \notag \\
& = - e^{-\frac{(s+1)^2}{L^2}} \frac{1 - e^{i(s+1)(\eta-\eta')}}{1 - e^{i(\eta-\eta')}} - \sum_{j=s+2}^{\infty}\left(e^{-\frac{j^2}{L^2}}-e^{-\frac{(j-1)^2}{L^2}}\right) \frac{1 - e^{ij(\eta-\eta')}}{1 - e^{i(\eta-\eta')}} \notag\\
& = - e^{-\frac{(s+1)^2}{L^2}} \frac{1 - e^{i(s+1)(\eta-\eta')}}{1 - e^{i(\eta-\eta')}} + e^{\frac{-(s+1)^2}{L^2}}\frac{1}{1 - e^{i(\eta-\eta')}} \notag \\
 & \qquad + \sum_{j=s+2}^{\infty}\left(e^{-\frac{j^2}{L^2}} -e^{-\frac{(j-1)^2}{L^2}}\right) \frac{ e^{ij(\eta-\eta')}}{1 - e^{i(\eta-\eta')}} \notag\\
& =  e^{-\frac{(s+1)^2}{L^2}} \frac{ e^{i(s+1)(\eta-\eta')}}{1 - e^{i(\eta-\eta')}} \label{estimate7a}\\
             &  \qquad + \sum_{j=s+2}^{\infty}\left(e^{-\frac{j^2}{L^2}}-e^{-\frac{(j-1)^2}{L^2}}\right) \frac{ e^{ij(\eta-\eta')}}{1 - e^{i(\eta-\eta')}}.\label{estimate7b}
\end{align}
Here we applied \eqref{e.sumbyparts} with
$$
a_j = e^{-\frac{j^2}{L^2}}, \quad b_j = \sum_{k = 0}^{j-1} e^{ik(\eta-\eta')} = \frac{1 - e^{ij(\eta-\eta')}}{1 - e^{i(\eta-\eta')}}, \quad J_\ell = s + 1, \quad J_r = \infty.
$$
Due to assumption (i) on $\nu$, we have
\begin{equation}\label{e.useof(i)}
\sup_j \left|\sum_{k=0}^{j}e^{ik(\eta-\eta')}\right| = \sup_j \left| \frac{1 - e^{ij(\eta-\eta')}}{1 - e^{i(\eta-\eta')}} \right| \lesssim \frac{1}{|\eta-\eta'|} \text{ for } \eta, \eta' \in \supp \nu.
\end{equation}

Let us first look at the contribution coming from \eqref{estimate7a}. So, insert \eqref{estimate7a} into \eqref{estimate6b2}, and then plug the resulting term into \eqref{estimate5}. We get
\begin{align}
 C \iint \, & d\nu(\eta) d\nu(\eta') f(\eta) \overline{f(\eta')}e^{i(\eta-\eta')} \sum_{s=0}^{\infty} e^{i(\frac{1}{\eta}- \frac{1}{\eta'} )  \sum_{k=0}^{s-1} |\alpha_k|^2}  \times \label{e.insert7a}\\
& \quad \quad \times \Big(\frac{1}{\eta}-\frac{1}{\eta'}\Big) |\alpha_s|^2 e^{-\frac{(s+1)^2}{L^2}} \frac{ e^{i(s+1)(\eta-\eta')}}{1 - e^{i(\eta-\eta')}}.\notag
\end{align}

Split the double integral in \eqref{e.insert7a} over $(\eta,\eta')$ into two parts over
\[
|\eta-\eta'| \le L^{\frac{\gamma}{D}-1} \textrm{ and }  |\eta-\eta'| > L^{\frac{\gamma}{D}-1}.
\]

Let us first consider the part over $|\eta-\eta'| \le L^{\frac{\gamma}{D}-1}$. In this case, using Taylor expansion, the absolute value of this expression can be estimated as follows,
\begin{align}
& \iint_{|\eta-\eta'| \le L^{\frac{\gamma}{D}-1}} \frac{| f(\eta) \overline{f(\eta')}|}{|\eta-\eta'|} \sum_{s=0}^{\infty}\left|\frac{1}{\eta} -\frac{1}{\eta'} \right| |\alpha_s|^2   e^{-\frac{(s+1)^2}{L^2}} \,  d\nu(\eta) \, d\nu(\eta') \notag \\
& = \iint_{|\eta-\eta'| \le L^{\frac{\gamma}{D}-1}} \frac{| f(\eta) \overline{f(\eta')}|}{|\eta\eta'|} \sum_{s=0}^{\infty} |\alpha_s|^2   e^{-\frac{(s+1)^2}{L^2}} \,  d\nu(\eta) \, d\nu(\eta')\notag \\
& \lesssim \iint_{|\eta-\eta'| \le L^{\frac{\gamma}{D}-1}} | f(\eta) \overline{f(\eta')}| ( 2 + \sum_{s=1}^{\infty} (1+s)^{-\frac{\gamma}{2}} |\alpha_{s+1}|^2 (1+s)^{\frac{\gamma}{2}}   e^{-\frac{(s+1)^2}{L^2}}) \,  d\nu(\eta) \, d\nu(\eta') \notag \\
& \lesssim \iint_{|\eta-\eta'| \le L^{\frac{\gamma}{D}-1}} | f(\eta) \overline{f(\eta')}| \,  d\nu(\eta) \, d\nu(\eta')  \notag \\
& \qquad + \iint_{|\eta-\eta'| \le L^{\frac{\gamma}{D}-1}} | f(\eta) \overline{f(\eta')}| \sum_{s=1}^{\infty} (1+s^2)^{-\frac{\gamma}{4}} |\alpha_{s+1}|^2 (1+s)^{\frac{\gamma}{2}}   e^{-\frac{(s+1)^2}{L^2}} \,  d\nu(\eta) \, d\nu(\eta') \notag \\
& \le \iint_{|\eta-\eta'| \le L^{\frac{\gamma}{D}-1}} | f(\eta) |^2 \,  d\nu(\eta) \, d\nu(\eta')  \notag \\
& \qquad + C\iint_{|\eta-\eta'| \le L^{\frac{\gamma}{D}-1}} | f(\eta) \overline{f(\eta')}| \left(\sum_{s=1}^{\infty} |\alpha_{s+1}|^4 (1+s)^{\gamma}\right)^{\frac{1}{2}} \times \notag\\
& \qquad \quad \times \left(\sum_{s=1}^{\infty} (1+s^2)^{-\frac{\gamma}{2}}e^{-\frac{2(s+1)^2}{L^2}} \right)^{\frac{1}{2}}\,  d\nu(\eta) \, d\nu(\eta') \notag \\
& \le \iint_{|\eta-\eta'| \le L^{\frac{\gamma}{D}-1}} | f(\eta) |^2 \,  d\nu(\eta) \, d\nu(\eta')  \notag \\
& \qquad + C\iint_{|\eta-\eta'| \le L^{\frac{\gamma}{D}-1}} | f(\eta) \overline{f(\eta')}| \left(\sum_{s=1}^{\infty} (1+s^2)^{-\frac{\gamma}{2}}e^{-\frac{2(s+1)^2}{L^2}} \right)^{\frac{1}{2}}\,  d\nu(\eta) \, d\nu(\eta'), \label{e.1+salpha}
\end{align}
where we used the Cauchy-Schwarz inequality and \eqref{e.vcass}.

We can estimate the summation in \eqref{e.1+salpha} by
\begin{align*}
\sum_{s=1}^{\infty} & (1+s^2)^{-\frac{\gamma}{2}}e^{-\frac{2(s+1)^2}{L^2}}\\
& = 2^{-\frac{\gamma}{2}}e^{-\frac{8}{L^2}} + \sum_{s=2}^{\infty} (1+s^2)^{-\frac{\gamma}{2}}e^{-\frac{2(s+1)^2}{L^2}} \\
& \le 1 +\sum_{s=2}^{\infty} s^{-\gamma} e^{-\frac{s^2}{L^2}} \\
& \le  1 + \int_1^{\infty} x^{-\gamma} e^{-\frac{x^2}{L^2}} \, dx \\
& \le 1 + \int_1^{L}  x^{-\gamma} e^{-\frac{x^2}{L^2}} \, dx + \int_L^{\infty} x^{-\gamma} e^{-\frac{x^2}{L^2}} \, dx \\
& \le  1 + \int_1^{L}  x^{-\gamma} \, dx + L^{-\gamma}\int_L^{\infty} e^{-\frac{x^2}{L^2}} \, dx  \\
& \le  1 +  \frac{1}{1-\gamma}x^{1-\gamma}|_{1}^{L} \, dx + L^{-\gamma}\int_L^{\infty} e^{-\frac{x^2}{L^2}} \, dx  \\
& \le 1 +  \frac{1}{1-\gamma} L^{1-\gamma}  + L^{-\gamma}\int_0^{\infty} e^{-\frac{x^2}{L^2}} \, dx \\
& =  1 +  \frac{1}{1-\gamma} L^{1-\gamma}  + L^{-\gamma} \frac{L\sqrt{\pi}}{2}  \\
& \le 1 + CL^{1-\gamma} \\
& \lesssim L^{1-\gamma}.
\end{align*}

Thus, \eqref{e.1+salpha} can be estimated by (a constant times)
\begin{align*}
L^{\frac{1-\gamma}{2}} & \iint_{|\eta-\eta'| \le L^{\frac{\gamma}{D}-1}} \frac{|f(\eta) \overline{f(\eta')}|}{|\eta\eta'|} \, d\nu(\eta) \, d\nu(\eta')\\
& \le L^{1-\gamma} \left(\iint_{|\eta-\eta'| \le L^{\frac{\alpha}{D}-1}} \frac{|f(\eta)|^2}{|\eta\eta'|} \, d\nu(\eta) \, d\nu(\eta') \right)^{\frac{1}{2}}\\
& \qquad \times \left(\iint_{|\eta-\eta'| \le L^{\frac{\gamma}{D}-1}} \frac{|f(\eta')|^2}{|\eta\eta'|} \, d\nu(\eta) \, d\nu(\eta') \right)^{\frac{1}{2}}\\
& = L^{1-\gamma} \iint_{|\eta-\eta'| \le L^{\frac{\gamma}{D}-1}} \frac{|f(\eta)|^2}{|\eta\eta'|} \, d\nu(\eta) \, d\nu(\eta') \\
& \le L^{1-\gamma} \int |f(\eta)|^2 \left( \int_{|\eta-\eta'| \le L^{\frac{\gamma}{D}-1}} d\nu(\eta') \right) \, d\nu(\eta)\\
& \lesssim L^{1-\gamma}L^{D(\frac{\gamma}{D}-1)} \int |f(\eta)|^2 \, d\nu(\eta) \\
& \le L^{1-D} \int |f(\eta)|^2 \, d\nu(\eta).
\end{align*}
Here we used the Cauchy-Schwarz inequality in the first step, symmetry in the second step, assumption (i) on $\nu$ in the third step, and assumption (ii) on $\nu$ in the fourth step.

Now we analyze the corresponding part of the double integral in \eqref{e.insert7a} over $|\eta-\eta'| > L^{\frac{\gamma}{D}-1}$. We have
\begin{align}
C & \iint_{|\eta-\eta'| > L^{\frac{\gamma}{D}-1}} \,  d\nu(\eta) d\nu(\eta') f(\eta) \overline{f(\eta')}e^{i(\eta-\eta')} \sum_{s=0}^{\infty} e^{i(\frac{1}{\eta}- \frac{1}{\eta'} )  \sum_{k=0}^{s-1} |\alpha_k|^2}  \times \notag\\
& \quad \quad \times \Big(\frac{1}{\eta}-\frac{1}{\eta'}\Big) |\alpha_s|^2 e^{-\frac{(s+1)^2}{L^2}} \frac{ e^{i(s+1)(\eta-\eta')}}{1 - e^{i(\eta-\eta')}} \notag\\
& = C \iint_{|\eta-\eta'| > L^{\frac{\gamma}{D}-1}} \,  d\nu(\eta) d\nu(\eta') \Big(\frac{1}{\eta}-\frac{1}{\eta'}\Big)f(\eta) \overline{f(\eta')} \frac{e^{i(\eta-\eta')}}{1 - e^{i(\eta-\eta')}}  \times\label{e.insert7aa}\\
& \quad \quad \times \sum_{s=0}^{\infty} e^{i(\frac{1}{\eta}- \frac{1}{\eta'} )  \sum_{k=0}^{s-1} |\alpha_k|^2}  |\alpha_s|^2 e^{-\frac{(s+1)^2}{L^2}}  e^{i(s+1)(\eta-\eta')}. \notag
\end{align}
Let
\[
a_s=e^{i(\frac{1}{\eta}- \frac{1}{\eta'} )   \sum_{k=0}^{s-1} |\alpha_k|^2}|\alpha_s|^2, \quad  b_s=e^{-\frac{(s+1)^2}{L^2}}  e^{i(s+1)(\eta-\eta')},\quad J_{\ell} = 0, \quad J_r = \infty.
\]
Performing a summation by parts \eqref{e.sumbyparts2} on the $s$-summation in \eqref{e.insert7aa}, we obtain
\begin{align}
\sum_{s=0}^{\infty} & e^{i(\frac{1}{\eta}- \frac{1}{\eta'} )  \sum_{k=0}^{s-1} |\alpha_k|^2} |\alpha_s|^2 e^{-\frac{(s+1)^2}{L^2}}  e^{i(s+1)(\eta-\eta')} \notag\\
& = |\alpha_0|^2 \sum_{s=0}^{\infty} e^{-\frac{(s+1)^2}{L^2}}  e^{i(s+1)(\eta-\eta')} \label{e.analog9a}\\
&\qquad + \sum_{s=0}^{\infty}\left(e^{i(\frac{1}{\eta}- \frac{1}{\eta'} )  \sum_{k=0}^{s} |\alpha_k|^2} |\alpha_{s+1}|^2 - e^{i(\frac{1}{\eta}- \frac{1}{\eta'} )  \sum_{k=0}^{s-1} |\alpha_k|^2} |\alpha_s|^2\right) \times \label{e.analog9b}\\
& \qquad \quad\times\sum_{k=s+1}^{\infty}e^{-\frac{(k+1)^2}{L^2}}  e^{i(k+1)(\eta-\eta')}. \notag
\end{align}
Performing a summation by parts on \eqref{e.analog9a} and the $k$-summation in \eqref{e.analog9b}, we can get expressions similar to \eqref{estimate7a} and \eqref{estimate7b}. Note that due to \eqref{e.useof(i)}, we know the absolute value of \eqref{estimate7a} can be bounded by a constant times $\frac{1}{|\eta-\eta'|}$. Then the absolute value of \eqref{e.analog9a} and that of $k$-summation in \eqref{e.analog9b} can be estimated by a constant times $\frac{1}{|\eta-\eta'|}$. Now, we consider the remaining term in \eqref{e.analog9b}. Its absolute value can be estimated as follows,
\begin{align*}
\left|\sum_{s=0}^{\infty} \right. & \left. \left(e^{i(\frac{1}{\eta}- \frac{1}{\eta'} )  \sum_{k=0}^{s} |\alpha_k|^2} |\alpha_{s+1}|^2-e^{i(\frac{1}{\eta}- \frac{1}{\eta'} )  \sum_{k=0}^{s-1} |\alpha_k|^2} |\alpha_s|^2\right)\right|\\
& \le 2 \sum_{s=0}^{\infty}  |\alpha_s|^2 \le 2|\alpha_0|^2 + \sum_{s=1}^{\infty} s^{\gamma} |\alpha_s|^2  \le C,
\end{align*}
where we used \eqref{e.vcass} in the last step.

Thus, the absolute value of \eqref{e.insert7aa} can be estimated as follows. For any $\varepsilon \in (0,D)$, we have
\begin{align*}
C & \iint_{|\eta-\eta'| > L^{\frac{\gamma}{D}-1}} \left|\frac{1}{\eta}-\frac{1}{\eta'}\right| |f(\eta) \overline{f(\eta')}| \frac{1}{|\eta-\eta'|^2} \, d\nu(\eta) d\nu(\eta') \\
& = C \iint_{|\eta-\eta'| > L^{\frac{\gamma}{D}-1}}  \frac{1}{|\eta\eta'|} |f(\eta) \overline{f(\eta')}| \frac{1}{|\eta-\eta'|} \, d\nu(\eta) d\nu(\eta') \\
& \lesssim \iint_{|\eta-\eta'| > L^{\frac{\gamma}{D}-1}}  \frac{|f(\eta) \overline{f(\eta')}|}{|\eta-\eta'|} \, d\nu(\eta) d\nu(\eta') \\
& \le \left(\iint_{|\eta-\eta'| > L^{\frac{\gamma}{D}-1}}  \frac{|f(\eta)|^2}{|\eta-\eta'|}\right)^{\frac{1}{2}} \, d\nu(\eta) d\nu(\eta') \\
& \qquad \times \left(\iint_{|\eta-\eta'| > L^{\frac{\gamma}{D}-1}}  \frac{|f(\eta')|^2}{|\eta-\eta'|}\right)^{\frac{1}{2}} \, d\nu(\eta) d\nu(\eta') \\
& = \iint_{|\eta-\eta'| > L^{\frac{\gamma}{D}-1}} \frac{|f(\eta)|^2}{|\eta-\eta'|} \, d\nu(\eta) d\nu(\eta') \\
& = \int |f(\eta)|^2  \int_{|\eta-\eta'| > L^{\frac{\gamma}{D}-1}}\frac{1}{|\eta-\eta'|^{1-D+\varepsilon} |\eta-\eta'|^{D-\varepsilon}} \, d\nu(\eta') d\nu(\eta)\\
& \le L^{-(\frac{\gamma}{D}-1)(1-D+\varepsilon)}\int |f(\eta)|^2  \int_{|\eta-\eta'| > L^{\frac{\gamma}{D}-1}}\frac{1}{|\eta-\eta'|^{D-\varepsilon}} \, d\nu(\eta')d\nu(\eta) \\
& = L^{(1 - \frac{\gamma}{D})(1-D+\varepsilon)}\int |f(\eta)|^2  \int_{|\eta-\eta'| > L^{\frac{\gamma}{D}-1}}\frac{1}{|\eta-\eta'|^{D-\varepsilon}} \, d\nu(\eta')d\nu(\eta)\\
& \lesssim L^{1-D} \int |f(\eta)|^2 \, d\nu(\eta),
\end{align*}
where we used the assumption (i) in the second step, Cauchy-Schwarz inequalities in the third step, symmetry in the forth setp and \cite[Lemma 2.2]{r2} in the last step.

\medskip

We now move on to studying \eqref{estimate7b}. Insert \eqref{estimate7b} into \eqref{estimate6b2}, and then plug the resulting term into \eqref{estimate5}. We get
\begin{align}
 C \iint \, & d\nu(\eta) d\nu(\eta') f(\eta) \overline{f(\eta')}e^{i(\eta-\eta')} \sum_{s=0}^{\infty} e^{i(\frac{1}{\eta}- \frac{1}{\eta'} )  \sum_{k=0}^{s-1} |\alpha_k|^2}  \times \label{e.insert7b}\\
& \quad \quad \times \Big(\frac{1}{\eta}-\frac{1}{\eta'}\Big) |\alpha_s|^2 \sum_{j=s+2}^{\infty} \left(e^{-\frac{j^2}{L^2}} -e^{-\frac{(j-1)^2}{L^2}} \right)\frac{ e^{ij(\eta-\eta')}}{1 - e^{i(\eta-\eta')}}.\notag
\end{align}
The same break-up as above will be used. If $|\eta-\eta'| \le L^{\frac{\gamma}{D}-1}$, note that
\begin{align*}
\left|\sum_{j=s+2}^{\infty} \right . & \left. \left(e^{-\frac{j^2}{L^2}} -e^{-\frac{(j-1)^2}{L^2}} \right)\frac{ e^{ij(\eta-\eta')}}{1 - e^{i(\eta-\eta')}}\right| \\
& \le \frac{1}{|\eta-\eta'|}\sum_{j=s+2}^{\infty} \left(e^{-\frac{(j-1)^2}{L^2}} - e^{-\frac{j^2}{L^2}}\right) \\
& = \frac{e^{-\frac{(s+1)^2}{L^2}}}{|\eta-\eta'|}.
\end{align*}
Thus, the absolute value of \eqref{e.insert7b} can be estimated by
\[
 C \iint_{|\eta-\eta'| \le L^{\frac{\gamma}{D}-1}} \, |f(\eta) \overline{f(\eta')}| \sum_{s=0}^{\infty} \frac{1}{|\eta\eta'|} |\alpha_s|^2e^{-\frac{(s+1)^2}{L^2}} d\nu(\eta) d\nu(\eta').
\]
Proceeding as above, we can get the desired bound.

If $|\eta-\eta'| > L^{\frac{\gamma}{D}-1}$, we perform the summation by parts \eqref{e.sumbyparts} on \eqref{estimate7b} one more time. Set
\[
a_j= e^{-\frac{j^2}{L^2}} - e^{-\frac{(j-1)^2}{L^2}}, \quad b_{j} = \frac{1-e^{ij(\eta-\eta')}}{1-e^{i(\eta-\eta')}},\quad  J_{\ell} = s+2, \quad J_{r} = \infty.
\]
Hence, we have
\begin{align*}
&\frac{ 1}{1 - e^{i(\eta-\eta')}}\sum_{j=s+2}^{\infty}  \left(e^{-\frac{j^2}{L^2}}-e^{-\frac{(j-1)^2}{L^2}}\right) e^{ij(\eta-\eta')} \\
& = - \frac{1}{1 - e^{i(\eta-\eta')}}\left((e^{-\frac{(s+2)^2}{L^2}} - e^{-\frac{(s+1)^2}{L^2}})\frac{1-e^{i(s+2)(\eta-\eta')}}{1-e^{i(\eta-\eta')}} \right.\\
& \qquad \left.+ \sum_{j=s+3}^{\infty} \frac{1-e^{ij(\eta-\eta')}}{1-e^{i(\eta-\eta')}} (e^{-\frac{j^2}{L^2}} - e^{-\frac{(j-1)^2}{L^2}}- e^{-\frac{(j-1)^2}{L^2}} + e^{-\frac{(j-2)^2}{L^2}})\right)\\
& = - \frac{1}{1 - e^{i(\eta-\eta')}}\left((e^{-\frac{(s+2)^2}{L^2}} - e^{-\frac{(s+1)^2}{L^2}})\frac{1-e^{i(s+2)(\eta-\eta')}}{1-e^{i(\eta-\eta')}} \right.\\
& \qquad \left.+ \sum_{j=s+3}^{\infty} \frac{1-e^{ij(\eta-\eta')}}{1-e^{i(\eta-\eta')}} (e^{-\frac{(j)^2}{L^2}} - 2e^{-\frac{(j-1)^2}{L^2}} + e^{-\frac{(j-2)^2}{L^2}})\right).
\end{align*}
Plugging the above expression into \eqref{e.insert7b}, the absolute value of \eqref{e.insert7b} can be estimated as follows since the absolute value of the above expression can be bounded by a constant times $\frac{1}{|\eta-\eta|^2}$.
\begin{align*}
C & \iint_{|\eta-\eta'| > L^{\frac{\gamma}{D}-1}} \, d\nu(\eta) d\nu(\eta') |f(\eta) \overline{f(\eta')}| \sum_{s=0}^{\infty}  \left|\frac{1}{\eta}-\frac{1}{\eta'}\right| |\alpha_s|^2 \frac{1}{|\eta-\eta'|^2} \\
& = C \iint_{|\eta-\eta'| > L^{\frac{\gamma}{D}-1}} \,  d\nu(\eta) d\nu(\eta') |f(\eta) \overline{f(\eta')}| \sum_{s=0}^{\infty}  |\alpha_s|^2 \frac{1}{|\eta-\eta'| |\eta\eta'|} \\
& \lesssim \iint_{|\eta-\eta'| > L^{\frac{\gamma}{D}-1}} \,  d\nu(\eta) d\nu(\eta') \frac{|f(\eta) \overline{f(\eta')}|}{|\eta-\eta'|},
\end{align*}
where we used the assumption (i) and \eqref{e.vcass} in the second step. Proceeding as above, one can get the desired bound.
\end{proof}

\begin{theorem}{\label{t3.1}}
Assume that \eqref{e.vcass} holds. Suppose that $\nu$ is a finite Borel measure on $(0,2\pi)$ with the following two properties:
\begin{itemize}

\item[{\rm (i)}] $\nu$ is supported by some compact subset $S_\nu$ of $(0,2\pi)$,

\item[{\rm (ii)}] $\nu$ is uniformly $D$-H\"older continuous for some $D \in (1-\gamma,1)$, that is, $\nu(I) \lesssim |I|^D$ for every interval $I \subseteq (0,2\pi)$.

\end{itemize}
Then there is a constant $C < \infty$ so that
$$
\sup_\beta \int \left| \sum_{s=x_{n-1}}^{x_n-1}  f(s) e^{i\omega(s,\eta,\beta)} \right|^2 \, d\nu(\eta) \le C (x_n-x_{n-1})^{1-D} \sum_{s=x_{n-1}}^{x_n-1} \left| f(s) \right|^2
$$
for all strictly increasing sequences $\{x_n\}_{n \in \Z_+} \subseteq \Z_+$ and all $f\in \ell^2(\Z_+)$.
\end{theorem}

\begin{proof}
Let us write
\[
g_n(\eta,\beta) = \sum_{s=x_{n-1}}^{x_n-1} f(s) e^{i\omega(s,\eta,\beta)},
\]
and $\phi_n(\eta,\beta)=\arg(g_n(\eta,\beta))$. Then,
\begin{align*}
\int  & |g_n(\eta,\beta)|^2 \, d\nu(\eta) = \int |g_n(\eta,\beta)| e^{-i\phi_n(\eta,\beta)} \sum_{s=x_{n-1}}^{x_n-1} f(s) e^{i\omega(s,\eta,\beta)} \, d\nu(\eta) \\
& = \sum_{s=x_{n-1}}^{x_n-1} f(s) \int |g_n(\eta,\beta)| e^{-i\phi_n(\eta,\beta)}e^{i\omega(s,\eta,\beta)} \, d\nu(\eta) \\
& = \sum_{s=0}^{x_n-1-x_{n-1}} f(s+x_{n-1}) \int |g_n(\eta,\beta)| e^{-i\phi_n(\eta,\beta)}e^{i\omega(s+x_{n-1},\eta,\beta)} \, d\nu(\eta) \\
& \le \|f\|_{\ell^2 (x_{n-1},x_n-1)} \times \\
& \qquad \times \left( \sum_{s=0}^{x_n-1-x_{n-1}} \left| \int |g(\eta,\beta)| e^{i(\omega_{\alpha_0}(x_{n-1}-1,\eta,0)-\phi(\eta,\beta))} e^{i\omega_{\alpha_{x_{n-1}}}(s,\eta,\beta)} \, d\nu(\eta) \right|^2 \right)^{\frac{1}{2}},
\end{align*}
where we used the Cauchy-Schwarz inequality in the last step and $\omega_{\alpha_n}(s,\eta,\beta) = (s+1)\eta +\beta +\frac{1}{\eta}\sum_{k=n}^{s+n-1}|\alpha_k|^2 $.
Now Theorem~\ref{t2.1} yields
\begin{equation}{\label{e.gnorm}}
\|g(\beta)\|_{L^2((0,2\pi), d\nu)}^2 \le C \|f\|_{\ell^2(x_{n-1},x_n-1)} (x_{n}-x_{n-1})^{\frac{1-D}{2}} \|g(\beta)\|_{L^2((0,2\pi), d\nu)}.
\end{equation}
Here, it is important that the constant $C$ is independent of the particular form of $\alpha$. Thus, since $\alpha_{\cdot + x_{n-1}}$ also satisfies \eqref{e.vcass} if $\alpha_\cdot$ does, the constant in \eqref{e.gnorm} is independent of $x_{n-1}$.

Divide \eqref{e.gnorm} by $\|g(\beta)\|_{L^2((0,2\pi), d\nu)}$ and then take squares. This concludes the proof.
\end{proof}

\section{The Iteration Procedure}\label{sec.4}

Our goal in this section is to prove the following theorem:

\begin{theorem}\label{t4.1}
Assume that \eqref{e.vcass} holds. Suppose that $\nu$ is a finite Borel measure on $(0,2\pi)$ with the following two properties:
\begin{itemize}

\item[{\rm (i)}] There is a $\delta > 0$ such that $\nu$ is supported by $(\delta,2\pi-\delta)$.

\item[{\rm (ii)}] There is a $D \in (1-\gamma,1)$ such that $\nu$ is uniformly $D$-H\"older continuous, that is, $\nu(I) \lesssim |I|^D$ for every interval $I \subseteq (0,2\pi)$.

\end{itemize}
Then
\begin{equation}\label{e.FinitePruferRadius}
\sup \{ R_n(\eta,\beta) : \beta \in [0, 2\pi), n \in \Z_+ \} < \infty
\end{equation}
for $\nu$-almost every $\eta$.
\end{theorem}

\begin{remark}\label{r.notfinitelysupp}
Note that \eqref{e.FinitePruferRadius} is trivially true, and indeed well known, for every $\eta$ if the sequence of Verblunsky coefficients is finitely supported. So let us assume that $\{ \alpha_n \}_{n \in \Z_+}$ is not finitely supported for the remainder of this section.
\end{remark}

The proof will work with a condition on the Verblunsky coefficients $\{ \alpha_n \}$ that is different from \eqref{e.vcass}, so let us first show that \eqref{e.vcass} actually implies the condition we will base the proof of Theorem~\ref{t4.1} on (namely \eqref{e.goal} and \eqref{e.goal2new} below).

\begin{lemma}\label{l.verification}
Given a sequence of Verblunsky coefficients $\{ \alpha_n \}_{n \in \Z_+} \subseteq \D$ that is not finitely supported, define the strictly increasing sequence $\{ x_n \}_{n \in \Z_+} \subseteq \Z_+$ by
\begin{itemize}
\item $x_0 = 0$,
\item for every $n \in \Z_+$, $x_{n+1}$ is the smallest power of $2$ so that $x_{n+1} > x_n$ and $\alpha_j\neq 0$ for at least one $j \in [x_n, x_{n+1})$.
\end{itemize}
If \eqref{e.vcass} holds, then for every
\begin{equation}\label{e.Dassumption}
D \in \left( 1 - \gamma, 1 \right)
\end{equation}
and every integer
\begin{equation}\label{e.Nassumptionnew}
N \in \left( \frac{2 - \gamma - D}{D + \gamma - 1}, \infty \right),
\end{equation}
we have
\begin{equation}\label{e.goal}
\sum_{n=1}^\infty |x_n  - x_{n-1}|^{\frac{1-D}{2}} \|\alpha \chi_{[x_{n-1},x_n)}\|_2 < \infty
\end{equation}
and
\begin{equation}\label{e.goal2new}
\sup_n \| \alpha \chi_{[x_{n-1},x_n)} \|_1 \| |x_n - x_{n-1}|^{\frac{1-D}{2}}  \alpha \chi_{[x_{n-1},x_n)} \|_2^N < \infty
\end{equation}
\end{lemma}

\begin{proof}
If \eqref{e.vcass} holds, then
$$
C := \sum_{s = 0}^\infty s^\gamma |\alpha_s|^2 < \infty.
$$

By construction of the sequence $\{ x_n \}_{n \in \Z_+}$ we can write, for $n \in \N$, $x_n=2^{m_n}$, for some strictly increasing sequence of positive integers $\{m_n\}$, and we have
\begin{align*}
\sum_{s = 2^{m_n}}^{2^{m_{n+1}}-1} s^\gamma |\alpha_s|^2 & = \sum_{s = 2^{m_{n+1} - 1}}^{2^{m_{n+1}}-1} s^\gamma |\alpha_s|^2 \\
& \ge (2^{m_{n+1} - 1})^\gamma \sum_{s = 2^{m_{n+1} - 1}}^{2^{m_{n+1}}-1} |\alpha_s|^2,
\end{align*}
which in turn yields
\begin{align}
\label{e.localL2estimate}
\| \alpha \chi_{[2^{m_n},2^{m_{n+1}})}\|_2 & = \| \alpha \chi_{[2^{m_{n+1} - 1},2^{m_{n+1}})}\|_2 \\
\nonumber & \le \sqrt{\frac{C}{2^{\gamma(m_{n+1} - 1)}}} \\
\nonumber & \lesssim 2^{-\frac{\gamma(m_{n+1}-1)}{2}}.
\end{align}
The Cauchy-Schwarz inequality together with \eqref{e.localL2estimate} then yields
\begin{align}
\label{e.localL1estimate}
\| \alpha \chi_{[2^{m_n},2^{m_{n+1}})}\|_1 & = \| \alpha \chi_{[2^{m_{n+1} - 1}, 2^{m_{n+1}})}\|_1 \\
\nonumber & \le 2^{\frac{m_{n+1}-1}{2}} \| \alpha \chi_{[2^{m_{n+1} - 1}, 2^{m_{n+1}})}\|_2 \\
\nonumber & \lesssim 2^{\frac{(1 - \gamma)(m_{n+1}-1)}{2}}.
\end{align}
Observe that
\begin{equation}\label{e.2powers}
|2^{m_{n+1}}-2^{m_n}|^{\frac{1-D}{2}}=\left|\sum_{j = m_n}^{m_{n+1} - 1} 2^j \right|^{\frac{1-D}{2}}\leq \sum_{j=m_n}^{m_{n+1} - 1} 2^{\frac{j(1-D)}{2}}
\end{equation}
and
\begin{align}
|2^{m_{n+1}}-2^{m_n}|^{\frac{(1-D)N}{2}} & = \left| \sum_{j = m_n}^{m_{n+1} - 1} 2^j \right|^{\frac{(1-D)N}{2}} \notag \\
& \le (m_{n+1} - m_n)^{\frac{(1-D)N}{2}} \sum_{j=m_n}^{m_{n+1} - 1} 2^{\frac{j(1-D)N}{2}} \notag \\
& \le 2^{\frac{(1-D)N}{2} \log_2(m_{n+1} - m_n)}\sum_{j=m_n}^{m_{n+1} - 1} 2^{\frac{j(1-D)N}{2}} \notag \\
& \le 2^{\frac{(1-D)N}{2} (m_{n+1} - 1)}\sum_{j=m_n}^{m_{n+1} - 1} 2^{\frac{j(1-D)N}{2}}, \label{e.2powers2}
\end{align}
where the second step holds since
$$
\left( \sum_{j=1}^{M} |a_j| \right)^{N} \le \left( M \max_{1\le j\le M} |a_j| \right)^N = M^N \max_{1\le j \le M} |a_j|^N \le M^N \sum_{j=1}^{M} |a_j|^{N}.$$

Let us now establish \eqref{e.goal} and \eqref{e.goal2new}. By \eqref{e.localL2estimate} and \eqref{e.2powers}, we have
\begin{align}
\left| 2^{m_{n+1}}  - 2^{m_n} \right|^{\frac{1-D}{2}} \| \alpha \chi_{[2^{m_n},2^{m_{n+1}})}\|_2 & \lesssim \sum_{j=m_n}^{m_{n+1} - 1} 2^{\frac{j(1-D)}{2}} 2^{-\frac{\gamma(m_{n+1}-1)}{2}} \notag \\
& \le \sum_{j=m_n}^{m_{n+1} - 1} 2^{\frac{j(1-D)}{2}} 2^{-\frac{j\gamma}{2}}\notag\\
& = \sum_{j=m_n}^{m_{n+1} - 1}2^{-\frac{j(\gamma + D - 1)}{2}},\label{e.localL2estimate*2}
\end{align}
which is summable when \eqref{e.Dassumption} holds, and \eqref{e.goal} follows.

Next, by \eqref{e.localL2estimate}, \eqref{e.localL1estimate} and \eqref{e.2powers2}, we have
\begin{align*}
& \| \alpha \chi_{[2^{m_n}, 2^{m_{n+1}})} \|_1  |2^{m_{n+1}}  - 2^{m_n}|^{\frac{(1-D)N}{2}} \| \alpha \chi_{[2^{m_n}, 2^{m_{n+1}})} \|_2^N \\
& \lesssim 2^{\frac{(1 - \gamma)(m_{n+1}-1)}{2}} 2^{\frac{(1-D)N}{2} (m_{n+1} - 1)} \left( \sum_{j=m_n}^{m_{n+1} - 1} 2^{\frac{j(1-D)N}{2}} \right) 2^{-\frac{\gamma N (m_{n+1}-1)}{2}} \\
& = 2^{\frac{m_{n+1} - 1}{2}(1 - \gamma + (1 - D- \gamma) N)} \left(\sum_{j=m_n}^{m_{n+1} - 1} 2^{\frac{j(1-D)N}{2}} \right)  \\
& \leq \sum_{j=m_n}^{m_{n+1} - 1} 2^{\frac{j}{2}(1 - \gamma + (1 - D- \gamma) N)}2^{\frac{j(1-D)}{2}} \text{ since $1 - \gamma + (1 - D- \gamma) N$ is negative  } \\
& = \sum_{j=m_n}^{m_{n+1} - 1} 2^{\frac{j}{2}[(1 - \gamma) + (1-D) + (1 - D- \gamma) N]} ,
\end{align*}
which is bounded when \eqref{e.Nassumptionnew} holds, and \eqref{e.goal2new} follows.
\end{proof}

Let us begin by proving a slightly weakened version of Theorem \ref{t4.1}. The only difference is that the supremum \eqref{e.FinitePruferRadius} is taken over the set $\{x_n\}$ rather than over $n \in \mathbb \Z_+$.

\begin{prop}\label{t4.1-weakened}
Assume that \eqref{e.vcass} holds, and let $\{x_n\}$ be chosen as in Lemma \ref{l.verification}. Suppose that $\nu$ is a finite Borel measure on $(0,2\pi)$ with the following two properties:
\begin{itemize}

\item[{\rm (i)}] There is a $\delta > 0$ such that $\nu$ is supported by $(\delta,2\pi-\delta)$.

\item[{\rm (ii)}] There is a $D \in (1-\gamma,1)$ such that $\nu$ is uniformly $D$-H\"older continuous, that is, $\nu(I) \lesssim |I|^D$ for every interval $I \subseteq (0,2\pi)$.

\end{itemize}
Then
\begin{equation}\label{e.FinitePruferRadius-weakened}
\sup \{ R_{x_n}(\eta,\beta) : \beta \in [0, 2\pi), n \in \Z_+ \} < \infty
\end{equation}
for $\nu$-almost every $\eta$.
\end{prop}

\begin{proof}
Recall from Remark~\ref{r.notfinitelysupp} that we work under the assumption that $\{ \alpha_n \}_{n \in \Z_+}$ is not finitely supported. Therefore, Lemma~\ref{l.verification} applies and hence \eqref{e.goal} and \eqref{e.goal2new} hold for the sequence $\{ x_n \}_{n \in \Z_+}$ constructed there and $N$ satisfying \eqref{e.Dassumption} and \eqref{e.Nassumptionnew}.

With the Pr\"ufer variables $R_n(\eta,\beta)$ and $\theta_n (\eta,\beta)$ from \eqref{e.pruferdef} and $\omega(s,\eta,\beta)$ from \eqref{e.omegadef}, we further define
\begin{equation}\label{e.tau}
\tau(n,\eta,\beta) = 2\theta_n (\eta,\beta) - \frac{1}{\eta} \sum_{j=0}^{n-1} |\alpha_j|^2,
\end{equation}
and
\begin{equation}\label{e.psi}
\psi(n,\eta,\beta)=\omega(n,\eta,\beta) + \tau(n,\eta,\beta).
\end{equation}
We note that, equivalently,
\begin{equation}\label{e.psi2}
\psi(n,\eta,\beta)=(n+1)\eta + \beta + 2 \theta_n(\eta,\beta).
\end{equation}

We will aim to generalize the iteration scheme of \cite{r, r2}. Our goal is to show that
\begin{equation}\label{e.FinitePruferRadius2}
\sup_\beta \sum_{n=1}^\infty |\ln R_{x_n}(\eta,\beta)-\ln R_{x_{n-1}}(\eta,\beta)| < \infty
\end{equation}
for $\nu$-almost every $\eta$, as \eqref{e.FinitePruferRadius} then follows. In the following, we will suppress the $\beta$-dependence of the various quantities from the notation and note that our arguments are uniform in $\beta$.

It suffices to show that for any $\nu$ satisfying assumptions (i) and (ii) with
\begin{equation}\label{e.Dcondition}
D \in (1-\gamma,1),
\end{equation}
we have that
$$
\sum_{n=1}^\infty \left\vert \sum_{j = x_{n-1}}^{x_n-1} \alpha_j e^{i\omega(j,\eta)} e^{i\tau(j,\eta)} \right\vert
$$
is finite for $\nu$-almost every $\eta$. We perform a summation by parts to obtain
\begin{align}
\sum_{j=x_{n-1}}^{x_n-1}\alpha_je^{i\omega(j,\eta)}e^{i\tau(j,\eta)} & = e^{i\tau(x_n,\eta)} \sum_{j = x_{n-1}}^{x_n-1} \alpha_j e^{i\omega(j,\eta)} \notag \\
& \quad + \sum_{k = x_{n-1}}^{x_n-1} (e^{i\tau(k,\eta)} - e^{i\tau(k+1,\eta)}) \sum_{j = x_{n-1}}^{k} \alpha_j e^{i\omega(j,\eta)}, \label{e.key}
\end{align}
where we need the summation by parts of the following form
\[
\sum_{k=0}^{n} f_k g_k = g_{n+1} \sum_{k=0}^n f_k + \sum_{k=0}^n (g_k - g_{k+1}) \sum_{j=0}^{k} f_j.
\]


Due to Theorem~\ref{t3.1}, we see that the first term is already summable for $\nu$-almost every $\eta$. Namely, we have that
\begin{align*}
\int \left|\sum_{s=x_{n-1}}^{x_n-1} \alpha_s e^{i\omega(s,\eta)} \right| \, d\nu(\eta) & \le C \left( \int \left|\sum_{s=x_{n-1}}^{x_n-1}\alpha_se^{i\omega(s,\eta)}\right|^2 \, d\nu(\eta) \right)^{\frac{1}{2}} \\
& \le C (x_n-x_{n-1})^{\frac{1-D}{2}} \left( \sum_{s=x_{n-1}}^{x_n-1} |\alpha_{s}|^2 \right)^{\frac{1}{2}},
\end{align*}
which is summable due to \eqref{e.goal}. Here we used the Cauchy-Schwarz inequality in the first step and Theorem~\ref{t3.1} in the second step.
So, it suffices to deal with the second term, \eqref{e.key}.

Looking back at \eqref{e.key}, to show the second term converges we need to expand out the $e^{i\tau(k+1, \eta)}-e^{i\tau(k, \eta)}$ term. Suppressing the dependence on $\eta$ from the notation, we observe that
\begin{align}
\notag e^{is\tau(k+1)} & - e^{is\tau(k)} \\
\notag& = e^{is\tau(k)} (e^{is(\tau(k+1) - \tau(k))}-1) \\
\notag& = e^{is\tau(k)} \left( e^{2is(\theta_{k+1} - \theta_{k})} e^{-\frac{is}{\eta} |\alpha_k |^2} - 1 \right) \\
\notag& = e^{is\tau(k)} \left( \left( \frac{1 + |\alpha_k|^2 - 2 \mathrm{Re} (\alpha_k e^{i\psi(k)})}{(1 - 2 \alpha_k e^{i\psi(k)} + \alpha_k^2 e^{2i\psi(k)})} \right)^s e^{-\frac{is}{\eta} |\alpha_k |^2} - 1 \right) \\
\notag& \approx e^{is\tau(k)} \left( \left( (1 - 2 \mathrm{Re} (\alpha_k e^{i\psi(k)})) (1 + 2 \alpha_k e^{i\psi(k)})  \right)^s e^{-\frac{is}{\eta} |\alpha_k |^2} - 1 \right) \\
\label{e.eistau}& \approx e^{is\tau(k)} \left( - 2 s \mathrm{Re} (\alpha_k e^{i\psi(k)}) + 2 s \alpha_k e^{i\psi(k)} \right).
\end{align}
Here, we use \eqref{e.tau} in the second step, \eqref{eq:ThetaPrufer} and \eqref{e.psi2} in the third step, and the $``\approx"$ in the fourth and fifth steps means that we ignore terms of power $|\alpha_k|^2$ or higher.

Let us justify why we can neglect those terms of order $|\alpha_k|^\zeta$, $\zeta\geq 2$. We define $w:=\mathrm{Re}(\alpha_k e^{i\psi(k)})$ and $y:=\mathrm{Im}(\alpha_k e^{i\psi(k)})$, and  then expand out the following expression in $w$ and $y$.
\begin{align}
& \left( \frac{1 + |\alpha_k|^2 - 2 \mathrm{Re} (\alpha_k e^{i\psi(k)})}{(1 - 2 \alpha_k e^{i\psi(k)} + \alpha_k^2 e^{2i\psi(k)})} \right)^s\\
=&\left( \frac{1 + w^2+y^2 - 2 w}{1 - 2 (w+iy) + (w+iy)^2} \right)^s\notag\\
=&\left( \frac{1 + w^2+y^2 - 2 w}{1 - 2w+w^2-y^2+i(-2y+2wy)} \right)^s\\
=&\left(\frac{w^4-4w^3+6w^2-4w-y^4+1+i2(w-1)y(w^2-2w+y^2+1)}{1+w^4+4w^3+2w^2y^2+2w^2-12wy^2-4w+y^4+6y^2}\right)^s\notag \\
=&\left(\frac{w^4-4w^3+6w^2-4w-y^4+1}{1+w^4+4w^3+2w^2y^2+2w^2-12wy^2-4w+y^4+6y^2}\right.\\
&+i\left.\frac{2(w-1)y(w^2-2w+y^2+1)}{1+w^4+4w^3+2w^2y^2+2w^2-12wy^2-4w+y^4+6y^2}\right)^s
\end{align}

If we expand this out as a multivariable Taylor series in $w$ and $y$, it is clear that it converges absolutely when
\begin{equation}
|w|^4+|4w^3|+|2w^2y^2|+|2w^2|+|-12wy^2|+|-4w|+|y^4|+|6y^2|<1
\end{equation}
It is clear that in the $wy$-plane this region contains a disk of positive radius around the origin. Let's call this radius $\mathcal R$.

Phrased differently, we know that both the Taylor series

\begin{equation}\label{e.Taylor}
\mathrm{Re}\left(\left( \frac{1 + |\alpha_k|^2 - 2 \mathrm{Re} (\alpha_k e^{i\psi(k)})}{(1 - 2 \alpha_k e^{i\psi(k)} + \alpha_k^2 e^{2i\psi(k)})} \right)^s\right)=1+C_{1,1}w+C_{1,2}y+C_{2,1}w^2+C_{2,2}wy+C_{2,3}y^2+\ldots
\end{equation}
and
\begin{equation}\label{e.Taylor2}
\mathrm{Im}\left(\left( \frac{1 + |\alpha_k|^2 - 2 \mathrm{Re} (\alpha_k e^{i\psi(k)})}{(1 - 2 \alpha_k e^{i\psi(k)} + \alpha_k^2 e^{2i\psi(k)})} \right)^s\right)=1+D_{1,1}w+D_{1,2}y+D_{2,1}w^2+D_{2,2}wy+D_{2,3}y^2+\ldots
\end{equation}
converge absolutely whenever $w^2+y^2<\mathcal R^2$. Equivalently, both series converge absolutely whenever $|\alpha_k|<\mathcal R$.

We observe (for $\zeta\geq 2$ an integer, and $\Upsilon$ another integer that obeys $0\leq\Upsilon\leq \zeta$) that
\begin{align}
\int & \left|\sum_{n=1}^\infty \sum_{k=x_{n-1}}^{x_n-1}
\mathrm{Re}(\alpha_k e^{i\psi_k})^\Upsilon \mathrm{Im}(\alpha_k e^{i\psi_k})^{\zeta-\Upsilon}\sum_{j=x_{n-1}}^{k} \alpha_j e^{i\omega(j,\eta)}\right| \, d\nu(\eta)\label{e.Realpha.firstline}\\
& \leq \int \sum_{n=1}^\infty \sum_{k=x_{n-1}}^{x_n-1}
|\alpha_k|^\zeta\left| \sum_{j=x_{n-1}}^{k} \alpha_j e^{i\omega(j,\eta)}\right| \, d\nu(\eta) \\
& \leq C \sum_{n=1}^\infty \sum_{k=x_{n-1}}^{x_n-1} |\alpha_k|^\zeta \sqrt{\sum_{j=x_{n-1}}^{k} (k-x_{n-1})^{1-D}|\alpha_j|^2} \; \text{by Theorem \ref{t3.1} and Cauchy-Schwarz} \\
& \leq C \sum_{n=1}^\infty \sum_{k=x_{n-1}}^{x_n-1}
|\alpha_k|^\zeta \sqrt{\sum_{j=x_{n-1}}^{x_n-1} (x_n-1-x_{n-1})^{1-D}|\alpha_j|^2}\\
& \leq C \sum_{n=1}^\infty \left(\sum_{k=x_{n-1}}^{x_n-1}
|\alpha_k|^2 \right)^{\zeta/2} \sqrt{\sum_{j=x_{n-1}}^{x_n-1} (x_n-1-x_{n-1})^{1-D}|\alpha_j|^2} \\
& \leq C \sum_{n=1}^\infty(x_n-1-x_{n-1})^{\frac{1-D}{2}}\left( \sum_{k=x_{n-1}}^{x_n-1}
|\alpha_k|^2 \right)^{(\zeta+1)/2} \\
& \leq C \sum_{n=1}^\infty\left((x_n-1-x_{n-1})^{1-D} \sum_{k=x_{n-1}}^{x_n-1}
|\alpha_k|^2 \right)^{(\zeta+1)/2}
\end{align}

If we assume that $x_n$ is chosen as in the statement of Lemma \ref{l.verification} and $D\in(1-\gamma,1)$ we can use \eqref{e.localL2estimate*2} to deduce
\begin{align}
\notag \int & \left| \sum_{n=1}^\infty \sum_{k=x_{n-1}}^{x_n-1}
\mathrm{Re}(\alpha_k e^{i\psi_k})^\Upsilon \mathrm{Im}(\alpha_k e^{i\psi_k})^{\zeta-\Upsilon}\sum_{j=x_{n-1}}^{k} \alpha_j e^{i\omega(j,\eta)} \right| \, d\nu(\eta)\\
& \leq \sum_{n=1}^\infty 2^{-\frac{n(\gamma+D-1)}{2}(\zeta+1)}\leq \left(\sum_{n=1}^\infty 2^{-\frac{n(\gamma+D-1)}{2}}\right)^{\zeta+1}
\end{align}

So we bound this term by $K^{\zeta+1}$, for some constant $K$. Note however that by taking the sum as $\sum_{n=N}^\infty$ rather than $\sum_{n=1}^\infty$, if we take a large $N$ we may make our $K$ to be as small as we desire. In particular, we can make sure that $0<K<\min(\mathcal R,1)$.

Then at the end we get
\begin{align}
\int &\left|\sum_{n=N}^\infty \sum_{k=x_{n-1}}^{x_n-1} \left( \frac{1 + |\alpha_k|^2 - 2 \mathrm{Re} (\alpha_k e^{i\psi(k)})}{(1 - 2 \alpha_k e^{i\psi(k)} + \alpha_k^2 e^{2i\psi(k)})} \right)^s\sum_{j=x_{n-1}}^{k} \alpha_j e^{i\omega(j,\eta)}\right| \, d\nu(\eta) \notag \\
& \leq \text{ constant}+\int \left|\sum_{n=1}^\infty \sum_{k=x_{n-1}}^{x_n-1}
(\text{Terms that are linear in $\alpha_k$})\sum_{j=x_{n-1}}^{k} \alpha_j e^{i\omega(j,\eta)}\right| \, d\nu(\eta) \notag\\
& \qquad + |C_{2,1} K^2| +|C_{2,2}K^2| +|C_{2,3} K^2| + \ldots \notag \\
& \qquad + |C_{3,1}K^3|+|C_{3,2}K^3|+|C_{3,3}K^3| +|C_{3,4}K^3|\notag \\
& \qquad + \ldots\notag\\
& \qquad + |D_{2,1} K^2| + |D_{2,2} K^2| + |D_{2,3} K^2| + \ldots \notag \\
& \qquad + |D_{3,1} K^3| + |D_{3,2} K^3| + |D_{3,3}K^3| + |D_{3,4}K^3| \\
& \qquad + \ldots . \notag
\end{align}

And since $K<\mathcal R$ we are within the region of absolute convergence of both \eqref{e.Taylor} and \eqref{e.Taylor2}, so we are done. Thus we only need to worry about the linear terms in \eqref{e.eistau}.

So to show the second term of \eqref{e.key} converges, it suffices to show for $s=1$,
\begin{equation} \label{e.ignorealpha2}
\sum_{n=1}^\infty \sum_{k=x_{n-1}}^{x_n-1} e^{is\tau(k)}\left( -2s\mathrm{Re}(\alpha_ke^{i\psi(k)})+2s\alpha_ke^{i\psi(k)} \right)\sum_{j=x_{n-1}}^{k} \alpha_j e^{i\omega(j,\eta)}
\end{equation}
converges for $\nu$-almost every $\eta$.

By \eqref{e.psi} it suffices to show that
\begin{equation}\label{e.needtoconverge}
\sum_{n=1}^\infty\sum_{k=x_{n-1}}^{x_n-1}
\alpha_ke^{i\omega(k)}e^{is\tau(k)}\sum_{j=x_{n-1}}^{k} \alpha_j e^{i\omega(j,\eta)}
\end{equation}
converges for $\nu$-almost every $\eta$ and every $s\in\mathbb N$.

We note that in regard to $s$, \eqref{e.needtoconverge} are stronger than necessary. At this stage we only need them to be true for $s=2$. However, later we will perform an iterated inductive argument, and $s$ will increase by $1$ at each step of that iteration. The number of iterations required will be bounded in an $n$-independent way. Thus it is easier to just prove the relevant statements for all $s \in \mathbb N$.

Let us now write out the iteration scheme. We subdivide the interval $[x_{n-1}, x_n)$ into $N_n$ subintervals, by picking numbers $y_1(n,l)$ so that
\begin{equation}
x_{n-1} = y_1(n,0) \le y_1(n,1) \le \ldots \le y_1(n,N_n) = x_n.
\end{equation}

We then define the next generations of this iteration scheme inductively. Suppose that
$y_i(n,l_1,l_2,\ldots, l_i)$ has been defined for $i=1,2,\ldots, j-1$. Then for fixed $n \in \mathbb N$ and $l_i \in \{1,2,\ldots, N_n\}$ for $i \leq j-1$, we pick numbers $y_j(l_1,l_2,\ldots l_j)$ again with $l_j \in \{1,2,\ldots, N_n\}$ satisfying
\begin{align*}
y_{j-1} (l_1,\ldots, l_{j-2},l_{j-1}-1) & = y_{j}(l_1,\ldots, l_{j-2},l_{j-1}-1,0) \\
& \le y_{j}(l_1,\ldots, l_{j-2},l_{j-1}-1,1) \\
& \le \ldots \\
& \le y_{j}(l_1,\ldots, l_{j-2},l_{j-1}-1,N_n) \\
& = y_{j-1}(l_1,\ldots, l_{j-2},l_{j-1}).
\end{align*}
We will usually write $y_i$ so that only the last argument is explicit, the others are suppressed from the notation.

We will now prove the following lemma, which is an analogue of \cite[Lemma 2.3]{r} and shows how the iteration scheme reduces the problem at hand to the analogous problem on finer and finer partitions.

\begin{lemma}\label{l.mainlemma}
In order for
\begin{equation}\label{e.lemmaj-1}
\sum_{n=1}^\infty \sum^{N_n}_{l_1,l_2,\ldots, l_{j-1}=1} \left|\sum_{k=y_{j-1}(l_{j-1}-1)}^{y_{j-1}(l_{j-1})-1} \alpha_k e^{i\omega(k,\eta)} e^{i s\tau(k,\eta)} \sum_{m=y_{j-1}(l_{j-1}-1)}^{k} \alpha_m e^{i\omega(m,\eta)}\right|
\end{equation}
to converge for $\nu$-almost every $\eta$ and $s \in \mathbb N$, it suffices to show that
\begin{equation}\label{e.sufficestoshow}
\sum_{n=1}^\infty \sum^{N_n}_{l_1,l_2,\ldots, l_{j}=1}\left| \sum_{k=y_{j}(l_{j}-1)}^{y_{j}(l_{j})-1} \alpha_k e^{i\omega(k,\eta)} e^{is \tau(k,\eta)} \sum_{m=y_{j}(l_{j}-1)}^{k} \alpha_m e^{i\omega(m,\eta)}\right|
\end{equation}
converges for $\nu$-almost every $\eta$ and $s \in \mathbb N$.
\end{lemma}

\begin{proof}
Let us denote the restriction of the sequence $\alpha$ to $[x_{n-1},x_n)$ by $\alpha^{(n)}$, that is,
\begin{equation}
\alpha^{(n)}_m = \begin{cases} 0, \text{ when $m \notin [x_{n-1},x_n)$} \\ \alpha_m \text{ when $m\in [x_{n-1},x_n)$}
\end{cases}
\end{equation}
We set
\begin{equation}\label{e.Nndef}
N_n := \max \left\{ 1, \left\lfloor \frac{1}{\sqrt{\sum_{m = x_{n-1}}^{x_n-1} (x_n  - x_{n-1})^{1-D} |\alpha_m|^2}} \right\rfloor \right\}.
\end{equation}

From \eqref{e.goal} and \eqref{e.Nndef}, we see that
\begin{equation}\label{e.Nnbound}
\frac{1}{2} \leq N_n \sqrt{\sum_{m = x_{n-1}}^{x_n-1}|x_n  - x_{n-1}|^{1-D}| \alpha_m|^2} \leq 1,
\end{equation}
provided that $n$ is sufficiently large.

Let us analyze part of the sum from \eqref{e.lemmaj-1}. Given a fixed choice of $l_1,l_2,\ldots, l_{j-1}$, note that
\begin{align}
\sum_{k=y_{j-1}(l_{j-1}-1)}^{y_{j-1}(l_{j-1})-1} & \alpha_k e^{i\omega(k,\eta)} e^{is\tau(k,\eta)} \sum_{m=y_{j-1}(l_{j-1}-1)}^{k} \alpha_m e^{i\omega(m,\eta)} \notag \\
& = \sum^{N_n}_{l_{j}=1} \sum_{k=y_{j}(l_{j}-1)}^{y_{j}(l_{j})-1} \alpha_k e^{i\omega(k,\eta)} e^{is\tau(k,\eta)} \sum_{m=y_{j-1}(l_{j-1}-1)}^{k} \alpha_m e^{i\omega(m,\eta)} \notag \\
& = \sum^{N_n}_{l_{j}=1} \sum_{k=y_{j}(l_{j}-1)}^{y_{j}(l_{j})-1} \alpha_k e^{i\omega(k,\eta)} e^{is\tau(k,\eta)} \times \left( \sum_{m=y_{j-1}(l_{j-1}-1)}^{y_{j}(l_{j}-1)-1} \alpha_m e^{i\omega(m,\eta)} \right. \notag \\
& \qquad + \left. \sum_{m=y_{j}(l_{j}-1)}^{k} \alpha_me^{i\omega(m,\eta)}\right) \notag \\
& = \sum^{N_n}_{l_{j}=1} \sum_{k=y_{j}(l_{j}-1)}^{y_{j}(l_{j})-1} \alpha_k e^{i\omega(k,\eta)} e^{is\tau(k,\eta)} \sum_{m=y_{j-1}(l_{j-1}-1)}^{y_{j}(l_{j}-1)-1} \alpha_m e^{i\omega(m,\eta)} \label{e.recursion1st}\\
& \qquad + \sum^{N_n}_{l_{j}=1} \sum_{k=y_{j}(l_{j}-1)}^{y_{j}(l_{j})-1} \alpha_k e^{i\omega(k,\eta)} e^{is\tau(k,\eta)} \sum_{m=y_{j}(l_{j}-1)}^{k} \alpha_m e^{i\omega(m,\eta)}. \label{e.recursion2nd}
\end{align}
The second part of the last expression, \eqref{e.recursion2nd}, is the same form as \eqref{e.needtoconverge} except the summation is on a smaller region. The other term, \eqref{e.recursion1st}, can be treated with the same summation by parts argument as above. We have
\begin{align}
\sum_{k=y_{j}(l_{j}-1)}^{y_{j}(l_{j})-1} & \alpha_k e^{i\omega(k,\eta)} e^{is\tau(k,\eta)} \notag \\
& = e^{is\tau(y_j(l_{j}))} \sum_{\ell=y_j(l_{j}-1)}^{y_j(l_{j})} \alpha_\ell e^{i\omega(\ell,\eta)} \label{e.S} \\
& \qquad - \sum_{k=y_j(l_j-1)}^{y_j(l_j)-1} (e^{is\tau(k+1,\eta)} - e^{is\tau(k,\eta)}) \sum_{\ell=y_j(l_j-1)}^{k} \alpha_\ell e^{i\omega(\ell,\eta)}. \label{e.lastterm}
\end{align}
If we plug in the first part of the sum, \eqref{e.S}, into \eqref{e.recursion1st}, and then integrate with respect to $\nu$, we obtain
\begin{align*}
\sum^{N_n}_{l_{j}=1}\int_0^{2\pi} & \left| e^{is\tau(y_j(l_j))}\sum_{\ell=y_j(l_j-1)}^{y_j(l_j)-1}\alpha_\ell e^{i\omega(\ell,\eta)}\sum_{m=y_{j-1}(l_{j-1}-1)}^{y_{j}(l_{j}-1)-1} \alpha_m e^{i\omega(m,\eta)}\right| \, d\nu(\eta) \\
& \leq \sum^{N_n}_{l_{j}=1} \int_0^{2\pi} \left\vert \sum_{\ell=y_j(l_j-1)}^{y_j(l_j)-1} \alpha_\ell e^{i\omega(\ell,\eta)} \right\vert \left|\sum_{m=y_{j-1}(l_{j-1}-1)}^{y_{j}(l_{j}-1)-1} \alpha_m e^{i\omega(m,\eta)}\right| \, d\nu(\eta) \\
& \lesssim C \sum^{N_n}_{l_{j}=1} (y_j(l_j) - y_j(l_j-1))^{\frac{1-D}{2}}\left\Vert\alpha_{\chi[y_j(l_j-1),y_j(l_j)-1]}\right\Vert_2 \\
& \qquad \qquad \times (y_{j}(l_{j}-1) - y_{j-1}(l_{j-1}-1))^{\frac{1-D}{2}}\left\Vert\alpha_{\chi[y_{j-1}(l_{j-1}-1),y_{j}(l_{j}-1)-1]}\right\Vert_2 \\
& \le (x_n-x_{n-1})^{\frac{1-D}{2}} \Vert  \alpha^{(n)} \Vert_2 \sum^{N_n}_{l_{j}=1} (y_j(l_j)-1 - y_j(l_j-1))^{\frac{1-D}{2}}\left\Vert \alpha_{\chi[y_j(l_j-1),y_j(l_j)-1]} \right\Vert_2 \\
& \le (x_n-x_{n-1})^{1-D}  \Vert \alpha^{(n)} \Vert_2^2 \sqrt{N_n} \\
& \le (x_n-x_{n-1})^{\frac{3(1-D)}{4}} \Vert  \alpha^{(n)} \Vert_2^{3/2} .
\end{align*}
Here we used Theorem \ref{t3.1} in the second step, the  Cauchy-Schwarz inequality in the fourth step, and \eqref{e.Nnbound} in the fifth step. This bound is summable.

Finally, we address the expression obtained by plugging \eqref{e.lastterm}  into \eqref{e.recursion1st}. We again expand out $e^{is\tau(k+1)}-e^{is\tau(k)}$ and ignore the terms of order $\alpha^2$ or higher, to obtain an expression of the form \eqref{e.ignorealpha2}.

In other words, we have that the modulus of this term is bounded by
\begin{align*}
&\sum^{N_n}_{l_{j}=1} \left|\sum_{k=y_j(l_j-1)}^{y_j(l_j)-1} (e^{is\tau(k)} - e^{is\tau(k-1)}) \sum_{\ell=y_j(l_j-1)}^{k} \alpha_\ell e^{i\omega(\ell,\eta)} \sum_{m=y_{j-1}(l_{j-1}-1)}^{y_{j}(l_{j}-1)-1} \alpha_m e^{i\omega(m,\eta)} \right| \\
& \lesssim \sum^{N_n}_{l_{j}=1} \left| \sum_{m=y_{j-1}(l_{j-1}-1)}^{y_{j}(l_{j}-1)-1} \alpha_m e^{i\omega(m,\eta)}\right| \times \left| \sum_{k=y_j(l_j-1)}^{y_j(l_j)-1} e^{i(s+1)\tau(k,\eta)} \alpha_k e^{i\omega(k,\eta)} \sum_{\ell=y_j(l_j-1)}^{k} \alpha_\ell e^{i\omega(\ell,\eta)} \right|.
\end{align*}


It suffices to show the $\nu$-almost everywhere boundedness of the maximal function
\begin{equation}
M_n(\eta) = \max_{l_{j}=1,2,\ldots,N_n} \left| \sum_{m=y_{j-1}(l_{j-1}-1)}^{y_{j}(l_{j}-1)-1} \alpha_m e^{i\omega(m,\eta)} \right|.
\end{equation}
In fact, an even stronger statement is true. Clearly,
\begin{align*}
 \int & M_n(\eta)^2 \, d\nu(\eta) \\
& \le \sum_{l_1,\ldots,l_{j}=1}^{N_n} \int \left| \sum_{m=y_{j-1}(l_{j-1}-1)}^{y_{j}(l_{j}-1)-1} \alpha_m e^{i\omega(m,\eta)} \right|^2 \, d\nu(\eta) \\
& \lesssim \sum_{l_1,\ldots,l_{j}=1}^{N_n} (y_{j}(l_{j}-1) - y_{j-1}(l_{j-1}-1))^{1-D}\left\Vert \alpha_{\chi[y_{j-1}(l_{j-1}-1),y_{j}(l_{j}-1)-1]} \right\Vert_2^2  \\
& \le \sum_{l_1,\ldots,l_{j}=1}^{N_n} (y_{j-1}(l_{j-1}) - y_{j-1}(l_{j-1}-1))^{1-D}\left\Vert \alpha_{\chi[y_{j-1}(l_{j-1}-1),y_{j-1}(l_{j-1})-1]} \right\Vert_2^2 \\
& \le {(x_n-x_{n-1})^{1-D}}\| \alpha^{(n)}\|_2^2 N_n \\
& \le C {(x_n-x_{n-1})^{\frac{1-D}{2}}} \|\alpha^{(n)}\|_2 \\
& \in \ell_1(\mathbb N),
\end{align*}
where we used Theorem \ref{t3.1} in the second step.

Thus, $\sum_n M_n(\eta)^2 < \infty$ for $\nu$-almost every $\eta$, which in particular implies that $M_n(\eta)$ is bounded for  $\nu$-almost every $\eta$, completing the proof.
\end{proof}

By \cite[Corollary 10.12.2]{S05}, we have
\begin{equation}{\label{e.absofdifftau}}
|\tau(n+1,\eta,\beta)-\tau(n,\eta,\beta)|\le C(|\alpha_n|+|\alpha_{n+1}|).
\end{equation}
Due to the mean value theorem and \eqref{e.absofdifftau}, the absolute value of \eqref{e.key} can be estimated by
\begin{align}
\sum_{k=x_{n-1}+1}^{x_n} & |\tau(k,\eta) - \tau(k-1,\eta)| \left| \sum_{j=x_{n-1}}^{k-1} \alpha_j e^{i\omega(j,\eta)} \right| \notag \\
& \le \sum_{k=x_{n-1}+1}^{x_n} (|\alpha_{k-1}| + |\alpha_k|) \left|\sum_{j=x_{n-1}}^{k-1} \alpha_j e^{i\omega(j,\eta)}\right|. \label{e.secpart}
\end{align}

Due to \eqref{e.secpart} and Lemma \ref{l.mainlemma} all that remains to do is to show that for suitably chosen $y_i$'s and for some $j$,
\begin{equation}\label{e.|alpha|sum}
\sum_{n=1}^\infty \sum_{l_1,l_2,\ldots, l_j=1}^{N_n} \sum_{\ell=y_j(l_j-1)}^{y_j(l_j)-1} |\alpha_\ell| \left| \sum_{t=y_j(l_j-1)}^\ell \alpha_t e^{i\omega(t,\eta)} \right| < \infty
\end{equation}
for $\nu$-almost every $\eta$. We then obtain \eqref{e.FinitePruferRadius2} for $\nu$-almost every $\eta$ and the proof will be finished.

The choice of these partitions is literally identical with Remling's choices in the Schr\"odinger case \cite[pp.~164--165]{r}, 
the only exceptions being that $\|V_n\|_1$ is replaced with $\|\alpha^{(n)}\|_1$,  $\|V_n\|_2$ is replaced with $|x_n - x_{n-1}|^{\frac{1-D}{2}}\|\alpha^{(n)}\|_2$ and $V_{nm}$ is replaced with $ |y_m-y_{m-1}|^{\frac{1-D}{2}} \alpha \chi_{(y_{m-1}, y_m-1)}$. For the benefit of the reader, we reproduce the algorithm for choosing the partitions here.

First, let us set aside the $y_j$ notation for the partition points for now. Instead of writing $y_j(l_1,l_2,\ldots, l_j)$ for the $N_n^j$ partition points, we instead write $z_0, z_1,\ldots z_{N_n^j}$ (where the $z_i$'s are in $\{x_{n-1}, x_{n-1}+1,\ldots ,x_n\}$). For an integer $m \in [0, N_n^j]$, we define
$$
\|\alpha_{n,m}\|_1 = \|\alpha \chi_{(z_{m-1}, z_m-1)}\|_1 = \sum_{j=z_{m-1}}^{z_m -1} |\alpha_j|
$$
and
$$
\|\alpha_{n,m}\|_2 = \|\alpha \chi_{(z_{m-1}, z_m-1)}\|_2 = \left(\sum_{j=z_{m-1}}^{z_m -1} |\alpha_j|^2 \right)^{\frac{1}{2}}.
$$
We first define $z_0=x_{n-1}$. We now proceed to define the other partition points inductively. Assume for some integer $m$ $z_0$, $z_1$, $z_2, \ldots z_{m-1}$ have been chosen. We then define $z_m$ to be the smallest integer strictly larger than $z_{m-1}$ so that $\|\alpha_{n,m}\|_1 > \|\alpha^{(n)}\|_1 N_n^{-3j/4}$ holds. If no such integer less than $x_n$ exists, we instead set $z_m=x_n$. Let us refer to this final index as $M$. Note that $M$ must satisfy $M\leq N_n^{3j/4}$. Since $N_n^j-(N_n^{3j/4}+1)> N_n^{3j/4}$ for large enough $n$, we can add the numbers $z_m-1$ for $m=1,2,\ldots, M$ to this collection of $z_i$'s. We then choose the remaining $z_i$'s arbitrarily.

As a result of our construction the $z_i$'s have the following property. Let us renumber the $z_i$'s we have chosen as $x_n=z_0\leq z_1\leq \ldots\leq z_{N_n^j}=x_n$. In this case, for each $m=1,2,\ldots M$, there are two possibilities:
\begin{itemize}
\item $\{z_{m-1},\ldots,z_m-1\}$ consists of a single point,
\item 
 $\|\alpha_{n,m}\|_1 \leq \|\alpha^{(n)}\|_1 N_n^{-3j/4}$.
\end{itemize}
It is possible that for some $m$ the set $\{z_{m-1},\ldots,z_m-1\}$ is empty, in which case we say the second possibility above holds. To obtain the original labelling in terms of this new labelling, we can set $y_i(n,l_1,\ldots,l_i)=z_m$, where
\begin{equation}
m=\sum_{r=1}^{i-1}(l_r-1)N_n^{j-r}+l_iN_n^{j-i}
\end{equation}
Now we integrate the summands of \eqref{e.|alpha|sum} with respect to $d\nu$. The result will be bounded by a constant times

\begin{equation}
\sum_{m=1}^{N_n^j}\|\alpha_{n,m}\|_1 (z_m - z_{m-1})^{\frac{1-D}{2}} \|\alpha_{n,m}\|_2
\end{equation}

We consider separately

\begin{itemize}
\item the sum over those $m$ for which $y_{m-1}=y_m-1$
\item the sum over those $m$ for which $y_{m-1}\neq y_m-1$
\end{itemize}
The first of those sums clearly does not exceed $(x_n-x_{n-1})^{1-D}\|\alpha^{(n)}\|_2^2$ and the second one can be estimated by
\begin{align}\label{e.Nnbound2}
N_n^{-3j/4}\|\alpha^{(n)}\|_1\sum_{m=1}^{N_n^j} (z_m- z_{m-1})^{\frac{1-D}{2}}\|\alpha_{n,m}\|_2 & \leq N_n^{-j/4}\|\alpha^{(n)}\|_1 (x_n - x_{n-1})^{\frac{1-D}{2}}\|\alpha^{(n)}\|_2 \\
& \leq C \|\alpha^{(n)}\|_1 \|(x_n - x_{n-1})^{\frac{1-D}{2}} \alpha^{(n)}\|_2^{j/4+1}. \notag
\end{align}
By \eqref{e.goal2new}, this is also summable provided $j\geq 4N$ (with $N$ from \eqref{e.Nassumptionnew}).
\end{proof}

\begin{remark}
Note that the last line of the proof implies that the iteration scheme in Lemma~\ref{l.mainlemma} only needs to be performed at most $4N$ times, where $N$ is an $n$-independent quantity.
\end{remark}


\begin{proof}[{Proof of Theorem \ref{t4.1}}]
Our goal is to show that \eqref{e.FinitePruferRadius-weakened} implies \eqref{e.FinitePruferRadius}. It suffices to prove that
\begin{equation}\label{e.interpolation}
\lim_{n\to\infty} \max_{\xi\in [x_{n-1}, x_n-1]}\left|\sum_{j=x_{n-1}}^\xi \alpha_j e^{i\omega(j,\eta)}e^{i\tau(j,\eta)} \right|=0
\end{equation}
for $\nu$-almost every $\eta$.

Since \eqref{e.vcass} holds with some $\gamma > 0$, we have $\frac{2}{1 + \gamma} < 2$, so that it is possible to choose some $p$ obeying
\begin{equation}\label{e.choiceofp}
\frac{2}{1 + \gamma} < p < 2.
\end{equation}
With this value of $p$, we find that
\begin{align*}
\sum_{n = 1}^\infty |\alpha_n|^p & = \sum_{n = 1}^\infty \left( n^{-\frac{p \gamma}{2}} \right) \left( n^{\frac{p \gamma}{2}}  |\alpha_n|^p \right) \\
& \le \left( \sum_{n = 1}^\infty \left( n^{-\frac{p \gamma}{2}} \right)^{(1 - \frac{p}{2})^{-1}} \right)^{1 - \frac{p}{2}} \left( \sum_{n = 1}^\infty \left( n^{\frac{p \gamma}{2}}  |\alpha_n|^p \right)^{\frac{2}{p}} \right)^{\frac{p}{2}} \\
& = \left( \sum_{n = 1}^\infty n^{-\frac{p \gamma}{2-p}} \right)^{1 - \frac{p}{2}} \left( \sum_{n = 1}^\infty n^\gamma  |\alpha_n|^2 \right)^{\frac{p}{2}} \\
& < \infty
\end{align*}
where we applied the H\"older inequality in the second step as well as $\frac{p \gamma}{2-p} > 1$ (which follows from \eqref{e.choiceofp}) and \eqref{e.vcass} in the final step.

Let $q$ to be the H\"older conjugate of $p$ (so that $1/p+1/q=1$). Then it follows by \cite{Kiselev} that the maximal function
$$
M_n(\eta)=\max_{\xi\in [x_{n-1}, x_n-1]}\left|\sum_{j=x_{n-1}}^\xi \alpha_je^{i\omega(j,\eta)} \right|.
$$
obeys
\begin{equation}\label{e.Kiselev}
\left(\int M_n(\eta)^q d\nu(\eta)\right)^{1/q}\leq C\left(\sum_{j=x_{n-1}}^{x_n-1}|\alpha_j|^p\right)^{1/p}.
\end{equation}

We now just repeat the steps from the proof of Proposition~\ref{t4.1-weakened}, except we are trying to prove convergence to zero, rather than summability. Let us define $\xi(\eta)$ so the maximum in \eqref{e.interpolation} is attained for $\xi=\xi(\eta)$. We use an integration by parts calculation similar to \eqref{e.key} to bound \eqref{e.interpolation} by

\begin{align}
&\left| e^{i\tau(x_n,\eta)} \sum_{j = x_{n-1}}^{\xi(\eta)} \alpha_j e^{i\omega(j,\eta)}\right| \notag \\
& \quad + \left|\sum_{k = x_{n-1}}^{\xi(\eta)} (e^{i\tau(k,\eta)} - e^{i\tau(k+1,\eta)}) \sum_{j = x_{n-1}}^{k} \alpha_j e^{i\omega(j,\eta)}\right|, \label{e.IBPinterpolate}
\end{align}
Let's call the first term $S(n,\eta)$. Using \eqref{e.Kiselev} we deduce that $\int S(n,\eta)d\nu(\eta)\leq  C\| \alpha^{(n)}\|^q_{\ell^p}$ which is summable because $q>p$. Thus $S(n,k)$ itself is in $\ell^q$ for $\nu$-almost every $\eta$, and so clearly goes to $0$ as $n\to\infty$.

We then have to consider the second term of \eqref{e.IBPinterpolate}. We make a slight modification to the iterated partition argument used in the proof of Proposition~\ref{t4.1-weakened}. We define $y(n,l)$ the same way as in that section, and define $L_1=L_1(n,\eta)$ so that $y_1(n,L_1-1)< \xi(\eta)\leq y_1(n,L_1)$.

Repeating the arguments of Section \ref{sec.4}, it then suffices to consider
\begin{align}&\sum_{l=1}^{L_1-1} \left|\sum_{k=y_1(l-1)}^{y_1(l)-1}  \alpha_k e^{i\omega(k)}e^{is\tau(k)}\sum_{j=x_{n-1}}^{k}\alpha_j e^{i\omega(j)}\right| \label{e.L1term1}\\
&+ \left|\sum_{k=y_1(L_1-1)}^{\xi(\eta)-1}  \alpha_k e^{i\omega(k)}e^{is\tau(k)}\sum_{j=x_{n-1}}^{k}\alpha_j e^{i\omega(j)}\right|. \label{e.L1term2}
  \end{align}
The \eqref{e.L1term1} term is identical to that considered in the proof of Proposition~\ref{t4.1-weakened}, and can be handled the same way. So we only have to consider the \eqref{e.L1term2} term. We repeat the steps that led us to obtain \eqref{e.recursion1st}, \eqref{e.recursion2nd}, \eqref{e.S}, and \eqref{e.lastterm}. Thus it now suffices for us to consider the three terms
\begin{align}
&  \left|\sum_{k=y_{j}(L_{1}-1)}^{\xi(\eta)-1} \alpha_k e^{i\omega(k,\eta)} e^{is\tau(k,\eta)} \sum_{m=y_{j}(L_{1}-1)}^{k} \alpha_m e^{i\omega(m,\eta)} \right|,\label{e.R98(23)}\\
&\left|\sum_{\ell=y_j(L_1-1)}^{\xi(\eta)-1} \alpha_\ell e^{i\omega(\ell,\eta)}\sum_{m=x_{n-1}}^{y_j(L_1-1)-1} \alpha_m e^{i\omega(m,\eta)}\right|\label{e.R98(21)}\\
&\left|\sum_{k=y_j(L_1-1)}^{\xi(\eta)-1} (e^{is\tau(k+1,\eta)} - e^{is\tau(k,\eta)}) \sum_{\ell=y_j(L_1-1)}^{k} \alpha_\ell e^{i\omega(\ell,\eta)}\right|\times \left|\sum_{m=x_{n-1}}^{y_j(L_1-1)-1} \alpha_m e^{i\omega(m,\eta)}\right|.
\label{e.R98(22)-pre}
\end{align}
Note that from our discussion in the proof of Proposition~\ref{t4.1-weakened}, to show that \eqref{e.R98(22)-pre} tends to zero for $\nu$-almost every $\eta$ it suffices to show that
\begin{equation}\label{e.R98(22)}
\left|\sum_{k=y_j(L_1-1)}^{\xi(\eta)-1} \alpha_k e^{i\omega(k,\eta)} e^{is\tau(k,\eta)}\sum_{\ell=y_j(L_1-1)}^{k} \alpha_\ell e^{i\omega(\ell,\eta)}\right|\times \left|\sum_{m=x_{n-1}}^{y_j(L_1-1)-1} \alpha_m e^{i\omega(m,\eta)}\right|
\end{equation}
 goes to $0$ for $\nu$-almost every $\eta$.

Using \eqref{e.Kiselev} we can see that \eqref{e.R98(21)} and the second multiplicand of \eqref{e.R98(22)} tend to zero for $\nu$-almost every $\eta$. So it suffices to show that terms of the form \eqref{e.R98(23)} go to zero for $\nu$-almost every $\eta$.

 We the perform the iterative partition procedure given by Lemma \ref{l.mainlemma} $j$ times. We define $L_2,L_3,\ldots L_j$ at each step of the procedure analogously to $L_1$. We also make explicit the $n$-dependence of the $y_j$. Now it suffices to show
\begin{equation}
\sum_{k=y_j(n;L_1,L_2,\ldots, L_{j-1}, L_j-1)}^{\xi(\eta)-1} \left|\alpha_k \right|\left|\sum_{\ell=y_j(n;L_1,L_2,\ldots, L_{j-1}, L_j-1)}^{k} \alpha_\ell e^{i\omega(\ell,\eta)}\right|
\end{equation}
converges to zero for $\nu$-almost every $\eta$. But we can bound this by

\begin{equation}\label{e.finalmax}
2M_n(\eta)\max_{l_1,\ldots,l_j}\sum_{k=y_j(n;l_1,l_2,\ldots, l_{j-1}, l_j-1)}^{y_j(n;l_1,l_2,\ldots, l_{j-1}, l_j)-1} \left|\alpha_k \right|.
\end{equation}
We already know that $M_n(\eta)$ tends to zero for $\nu$-almost every $\eta$.
The way the $y_j$ are chosen, there are two possibilities. Either $y_j(l_j-1)=y_j(l_j)-1$ or $\|\alpha_{n,l_j-1}\|_1\leq \|\alpha^{(n)}\|_1 N_n^{-3j/4}$. Taking the $\max$ in \eqref{e.finalmax} over all $l_j$ for which the first possibility holds results in an expression bounded by $\sup_n |\alpha_n|$. Assuming $j$ is chosen so $j\geq 4N$, we know that taking the $\max$ in \eqref{e.finalmax} over all $l_j$ for which the second possibility holds results in  a bounded expression by \eqref{e.Nnbound}.

In the second case, we have
\begin{align}
\max_{l_1, \ldots, l_j} \sum_{k = y_j(n;l_1,l_2,\ldots, l_{j-1}, l_j-1)}^{y_j(n;l_1,l_2,\ldots, l_{j-1}, l_j)-1} \left|\alpha_k \right| & \le \max_{l_1,\ldots,l_j}  \|\alpha_{n,l_j-1}\|_1 \notag \\
& \le \|\alpha^{(n)}\|_1 N_n^{-3j/4} \notag \\
& \le C(j) \|\alpha^{(n)}\|_1 \|(x_n - x_{n-1})^{\frac{1-D}{2}}\alpha^{(n)}\|_2^{\frac{3j}{4}}, \label{e.sup}
\end{align}
where we used \eqref{e.Nnbound} in the last step. By~\eqref{e.goal2new}, this expression is bounded provided $j\geq \frac{4N}{3}$.
\end{proof}

\section{Proof of the Main Theorem}\label{sec.5}

We are now ready to give the

\begin{proof}[Proof of Theorem~\ref{t.main}]
Under the condition \eqref{e.vcass}, we have to show that
$$
S = \{ \eta \in [0,2\pi) : \sup_{n \ge 0} \|T_n(e^{i\eta})\| = \infty \}
$$
has Hausdorff dimension at most $1 - \gamma$.

Assuming this fails, and hence $\dim_\mathrm{H}(S) > 1 - \gamma$, it follows from standard results in measure theory (as recalled in \cite{r2}; see also \cite{Falconer}) that there is a finite Borel measure $\nu$ with the following two properties:
\begin{itemize}

\item[{\rm (i)}] There is a $\delta > 0$ such that $\nu$ is supported by $S \cap (\delta,2\pi-\delta)$.

\item[{\rm (ii)}] There is a $D \in (1-\gamma, 1)$ such that $\nu$ is uniformly $D$-H\"older continuous, that is, $\nu(I) \lesssim |I|^D$ for every interval $I \subseteq (0,2\pi)$.

\end{itemize}
In particular, Theorem~\ref{t4.1} is applicable to this measure. That theorem, however, ensures for $\nu$-almost every $\eta$ the boundedness of the Pr\"ufer radius as $n \to \infty$ for all initial phases $\beta$, and in particular those two that correspond to the entries of the Szeg\H{o} matrices. In particular, for $\nu$-almost every $\eta$, the Szeg\H{o} matrices $T_n(e^{i\eta})$ remain bounded as $n \to \infty$, in contradiction with the definition of $S$ and the fact that $S$ supports the measure $\nu$. This completes the proof.
\end{proof}

\end{document}